%% file: mainr_3.tex
\documentclass{article}
\usepackage{amsmath,amsfonts,amssymb,amsthm,mathtools,bbm,color}
\usepackage{microtype}
\usepackage{lipsum}
\usepackage{hyperref}
\usepackage{multirow}
\hypersetup{colorlinks, linkcolor={red}, citecolor={blue}, urlcolor={black}}
\usepackage{comment}
\usepackage{tikz} 
\usepackage[shortlabels]{enumitem}
\usepackage{graphicx}
\usepackage{subcaption}
\usepackage{tkz-graph}
\usepackage{tabularx}
\usepackage{array}\newcolumntype{M}[1]{>{\centering\arraybackslash}m{#1}}
\newcolumntype{L}[1]{>{\raggedright\let\newline\\\arraybackslash\hspace{0pt}}m{#1}}
\newcolumntype{C}[1]{>{\centering\let\newline\\\arraybackslash\hspace{0pt}}m{#1}}
\newcolumntype{R}[1]{>{\raggedleft\let\newline\\\arraybackslash\hspace{0pt}}m{#1}}
\usepackage{makecell}

\usepackage[margin=1.25in]{geometry}

\allowdisplaybreaks

\newcommand\abs[1]{\left\lvert#1\right\rvert}

\newcommand\floor[1]{\left\lfloor#1\right\rfloor}
\DeclareMathOperator\forb{forb}
\DeclareMathOperator\Avoid{Avoid}

\newtheorem{thm}{Theorem}[section]
\newtheorem{lemma}[thm]{Lemma}
\newtheorem{prop}[thm]{Proposition}

\newtheorem{conj}[thm]{Conjecture}

\theoremstyle{definition}
\newtheorem{definition}[thm]{Definition}
\newtheorem{question}[thm]{Question}
\begin{document}
\title{An intermediate case of exponential multivalued forbidden matrix configurations}
\author{Wallace Peaslee\thanks{Supported by a UK Engineering and Physical Sciences Research Council Doctoral Training Partnership grant.}  \footnotemark[4]\\University of Cambridge\\Department of Applied Mathematics and Theoretical Physics\\Cambridge, United Kingdom
\and Attila Sali\thanks{Research was partially supported by the National Research, Development and Innovation Office (NKFIH) grants K--116769 and SNN-135643. This work was also supported by the BME-Artificial Intelligence FIKP grant of EMMI (BME FIKP-MI/SC) and by the Ministry of Innovation and Technology and the National Research, Development and Innovation Office within the Artificial Intelligence National Laboratory of Hungary.}  \footnotemark[4] \\HUN-REN Alfr\'ed R\'enyi Institute of Mathematics and\\Department of Computer Science, BUTE\\ Budapest, Hungary
\and Jun Yan\thanks{Supported by the Warwick Mathematics Institute Centre for Doctoral Training, and by funding from the UK EPSRC (Grant number: EP/W523793/1).}  \thanks{Research conducted under the auspices of the Budapest Semesters in Mathematics program.}\\University of Warwick\\ Mathematics Institute\\ Coventry, United Kingdom}
\maketitle
\begin{abstract}
The \textit{forbidden number} $\forb(m,F)$, which denotes the maximum number of distinct columns in an $m$-rowed $(0,1)$-matrix with no submatrix that is a row and column permutation of $F$, has been widely studied in extremal set theory. Recently, this function was extended to $r$-matrices, whose entries lie in $\{0,1,\cdots,r-1\}$. $\forb(m,r,F)$ is the maximum number of distinct columns in an $r$-matrix with no submatrix that is a row and column permutation of $F$. While $\forb(m,F)$ is polynomial in $m$, $\forb(m,r,F)$ is exponential for $r\geq 3$. Recently, $\forb(m,r,F)$ was studied for some small $(0,1)$-matrices $F$, and exact values were determined in some cases. In this paper we study $\forb(m,r,M)$ for $M=\begin{bmatrix}0&1\\0&1\\1&0\end{bmatrix}$, which is the smallest matrix for which this forbidden number is unknown. Interestingly, it turns out that this problem is closely linked with the following optimisation problem. For each triangle in the complete graph $K_m$, pick one of its edges. Let $m_e$ denote the number of times edge $e$ is picked. For each $\alpha\in\mathbb{R}$, what is $H(m,\alpha)=\max\sum_{e\in E(K_m)}\alpha^{m_e}$? We establish a relationship between $\forb(m,r,M)$ and $H(m,(r-1)/(r-2))$, find upper and lower bounds for $H(m,\alpha)$, and use them to significantly improve known bounds for $\forb(m,r,M)$. 
\end{abstract}

\section{Introduction}
For every matrix $A$, let $\abs{A}$ denote the number of columns in $A$. A matrix $A$ is called \textit{simple} if all of its $\abs{A}$ columns are pairwise distinct. We say a matrix $A$ contains a \textit{configuration} of a matrix $F$, and denote it by $F \prec A$, if some submatrix of $A$ is a row and column permutation of $F$. If $A$ does not contain a configuration of $F$, we say that $A$ \textit{avoids} $F$ and denote this as $F \not \prec A$.

An $r$-matrix is a matrix whose entries are contained in the set $\{0,1,\ldots,r-1\}$. Define $\Avoid(m,r,F)$ to be the set of all $m$-rowed simple $r$-matrices that avoid $F$. The forbidden configuration problem studies the following extremal function defined for a collection $\mathcal{F}$ of matrices. 
\begin{equation*}
    \forb(m,r,\mathcal{F}) = \max \left \{ \abs{A}\colon A\in\Avoid(m,r,F)\text{ for every }F\in\mathcal{F}\right \}
\end{equation*}
We say $A\in\Avoid(m,r,\mathcal{F})$ is \textit{extremal} if $|A|=\forb(m,r,\mathcal{F})$. When $r=2$ we usually write $\forb(m,{\mathcal F})$ in place of $\forb(m,2,{\mathcal F})$. We also use $\forb(m,r,F)$ instead of the more cumbersome $\forb(m,r,\{F\})$.

The basic result in the theory of forbidden configurations is Sauer's Theorem (also proved by Perles and Shelah \cite{Shelah}, and Vapnik and Chervonenkis \cite{VC}). 
\begin{thm}[\cite{Sauer}]\label{Sauer's Thm} For every positive integer $k$, let $K_k$ denote the $k\times 2^k$ matrix containing every possible $(0,1)$-column exactly once. Then for every positive integer $m$, we have
\[
\forb(m,K_k) = \binom{m}{k-1} + \binom{m}{k-2} + \cdots + \binom{m}{1} + \binom{m}{0}.
\]
\end{thm}
The matrix $K_k$ is also called the complete $2$-matrix. Note that for fixed $k$, $\forb(m,K_k)$ is polynomial in $m$. Alon \cite{ALON} proved a generalisation of this result for complete $r$-matrices $K_k^{r}$, which contain every possible $(0,1,\cdots,r-1)$-column exactly once. $\forb(m,r,K_k^r)$ turns out to be exponential in $m$ when $r>2$. This is a special case of a more general dichotomy phenomenon proved by F\"uredi and Sali \cite{FS}. In what follows, an \textit{$(i,j)$-matrix} is a matrix whose entries are either $i$ or $j$.

\begin{thm}[\cite{FS}]
Let ${\mathcal F}$ be a family of $r$-matrices. If for every pair $i,j\in\{0,1,\cdots,r-1\}$, there is an $(i,j)$-matrix in ${\mathcal F}$, then $\forb(m,r,{\mathcal F}) = O(m^k)$ for some positive integer $k$. If ${\mathcal F}$ has no $(i,j)$-matrix for some pair $i,j \in \{0,1,\cdots,r-1\}$, then $\forb(m,r,{\mathcal F}) = \Omega(2^m)$.
\end{thm}
The function $\forb(m,{\mathcal F})$ has been investigated extensively, and the reader is suggested to consult the dynamic survey of Anstee \cite{survey}. However, the investigation of the general case of $r$-matrices (for $r>2$) has just started. Earlier papers mainly focused on giving conditions that ensure polynomial bounds, such as  \cite{ANSTEE2023113492,ELLIS202024}. 

A result of Dillon and Sali \cite{DillonSali2021_ExponentialMultivalued} calculates $\forb(m,r, F)$ exactly for all $r\geq 3$ and all $3 \times 2$ and $3 \times 3$ simple $(0,1)$-matrices except for $M = \begin{bmatrix}0&1\\0&1\\1&0\end{bmatrix}$, which they bound by:
\begin{equation}\label{eq:trivi-lower-upper}
    m(r-1)^{m-1} + (r-1)^m \leq \forb(m, r, M) \leq \frac{3}{2}m(r-1)^{m-1} + (r-1)^m
\end{equation}
The lower bound in (\ref{eq:trivi-lower-upper}) follows from the observation that the matrix consisting of all columns with at most one 0 contains no configuration of $M$, while the upper bound follows from $\forb(m,M)=\floor{\frac{3m}2}+1$ and a general upper bound theorem of Dillon and Sali.
The aim of this paper is to study $\forb(m, r, M)$ and find better lower and upper bounds for them. 

It turns out that the value $\forb(m,r,M)$ is closely linked with the answer to the following optimisation problem. 

\begin{definition}\label{tcmdef}
A multigraph $\mathcal{G}$ on vertex set $V$ is called a \textit{triangular choice multigraph (TCM)} if it can be obtained as follows. Start with the empty graph on $V$. For every unordered triple $u,v,w$ of distinct vertices in $V$, choose exactly one of edge $uv,uw,vw$ and add a copy of it to the multigraph. In particular, the total number of edges in $\mathcal{G}$, counting multiplicities, is $\binom{|V|}{3}$.

A triangular choice multigraph $\mathcal{G}$ on $V$ is \textit{2-recursive} if $|V|=1,2$ and $\mathcal{G}$ is the empty graph, or $|V|\geq3$ and there exists a partition $V=V_1\cup V_2$, with $V_1,V_2\not=\emptyset$, and multigraphs $\mathcal{G}_1, \mathcal{G}_2, \mathcal{G}_3$ such that 
\begin{itemize}
\item $\mathcal{G}=\mathcal{G}_1\cup\mathcal{G}_2\cup\mathcal{G}_3$, where the union is taken with edge multiplicity.
\item $\mathcal{G}_1$ is a 2-recursive TCM on $V_1$ and $\mathcal{G}_2$ is a 2-recursive TCM on $V_2$.
\item $\mathcal{G}_3$ is obtained by adding a copy of edge $uv$ for every unordered triple $u,v,w\in V$ with $u,v\in V_1$ and $w\in V_2$, and adding a copy of edge $vw$ for every unordered triple $u,v,w\in V$ with $u\in V_1$ and $v,w\in V_2$.
\end{itemize}
\end{definition}

\begin{figure}[t]
\input{ExampleTCM}
\end{figure}

\begin{question}\label{triangleproblem}
Let $\mathcal{G}$ be a triangular choice multigraph on $[m]$. For every $ij\in[m]^{(2)}$, let $m^{\mathcal{G}}_{ij}$ denote the multiplicity of the edge $ij$, which is also the number of unordered triples $i,j,k$ in which the edge $ij$ is chosen. For every $\alpha\in\mathbb{R}$, define 
\begin{align*}
w(\mathcal{G},\alpha)&=\sum_{xy\in[m]^{(2)}}\alpha^{m^{\mathcal{G}}_{xy}},\\ 
H(m, \alpha)&=\max\left\{w(\mathcal{G},\alpha)\colon\mathcal{G}\text{ is a TCM on }[m]\right\},\\
H_{2}(m, \alpha)&=\max\left\{w(\mathcal{G},\alpha)\colon\mathcal{G}\text{ is a 2-recursive TCM on }[m]\right\}.    
\end{align*}
We say a TCM $\mathcal{G}$ is \textit{extremal} if, depending on context, it satisfies $w(\mathcal{G},\alpha)=H(m,\alpha)$ or $w(\mathcal{G},\alpha)=H_2(m,\alpha)$. What are the values of $H(m, \alpha)$ and $H_{2}(m,\alpha)$? What do the extremal graphs look like?
\end{question}

Our first result relates $\forb(m,r,M)$ to $H(m,\frac{r-1}{r-2})$ and $H_{2}(m,\frac{r-1}{r-2})$. 
\begin{thm}\label{main}
For every $m\geq0$, $r\geq3$,
$$H_{2}\left(m,\frac{r-1}{r-2}\right)\cdot(r-2)^{m-2}\leq\forb(m, r, M)-(r-1)^m-m(r-1)^{m-1}\leq H\left(m,\frac{r-1}{r-2}\right)\cdot(r-2)^{m-2}.$$
\end{thm}

Our second result is the following upper bound on $H(m,2)$, which we then use to find an upper bound for $\forb(m,r,M)$ for all $r\geq 3$.
\begin{thm}\label{thm:upperbds}
For every $m\geq 0$, $H(m,2)\leq\frac{83}{192}m2^{m-1}$.
\end{thm}
\begin{thm}\label{thm:generalupperbds}
For every $m\geq0$ and $r\geq 3$, we have $\forb(m,r,M)\leq\left(1+\frac{83}{192}\right)m(r-1)^{m-1}+(r-1)^m$.
\end{thm}

Our third result is a lower bound on $H_2(m,\alpha)$ obtained by a recursive construction, which then implies a lower bound on $\forb(m,r,M)$ for all $r\geq 3$. For every $\alpha>1$, let $\lambda(\alpha)=\sum_{j=1}^\infty\frac{2^{j-1}}{\alpha^{2^j}}$.
\begin{thm}\label{thm:generallowerbds}
For every $\alpha>1$, $\epsilon>0$, and all sufficiently large $m$, we have 
$$H_2(m,\alpha)>\frac12(\lambda(\alpha)-\epsilon)m\alpha^m.$$
Consequently, for every $r\geq 3$, $\epsilon>0$, and all sufficiently large $m$, we have
$$\forb(m,r,M)\geq\left(1+\frac{r-1}{2(r-2)^2}\lambda\left(\frac{r-1}{r-2}\right)-\epsilon\right)m(r-1)^{m-1}+(r-1)^m.$$
\end{thm}

Note that since $\lambda(2)\approx 0.390747$ and $\frac{83}{192}\approx0.432292$, Theorem \ref{main}, Theorem \ref{thm:upperbds} and Theorem \ref{thm:lowerbds} together improve the bounds for $\forb(m,3,M)$ given by (\ref{eq:trivi-lower-upper}). For every $r\geq 3$, one can compare to a Riemann sum to show that $\frac{r-1}{2(r-2)^2}\lambda(\frac{r-1}{r-2})>0.25$, and numerical experiments suggest a stronger bound $\frac{r-1}{2(r-2)^2}\lambda(\frac{r-1}{r-2})>0.36$.

Moreover, in the case when $\alpha=2$, we show that the construction we used in Theorem \ref{thm:generallowerbds} above to obtain a lower bound for $H_2(m,2)$ is actually optimal, which in particular provides an exact recurrence formula for $H_2(m,2)$. 
\begin{thm}\label{thm:lowerbds}
For every $m\geq 3$ except 6, let $k=k(m)$ be the unique integer such that $2^k+2^{k-1}\leq m<2^k+2^{k+1}$. For $m=6$, let $k=k(6)=3$. Then $$H_{2}(m,2)=H_{2}(2^k,2)2^{m-2^k}+H_{2}(m-2^k,2)2^{2^k}+2^k(m-2^k).$$
Consequently, $\lim_{m\to\infty}\frac{H_{2}(m,2)}{m2^{m-1}}=\lambda(2)=\sum_{j=1}^{\infty}\frac{2^{j-1}}{2^{2^j}}$.
\end{thm}

Furthermore, we have the following conjecture when $r=3$ and $\alpha=2$, .
\begin{conj}\label{conj}
$H(m,2)=H_{2}(m,2)$ for all $m\geq 3$. Consequently, $\forb(m,3,M)=2^m+m2^{m-1}+H(m,2)$, and
$$\lim_{m\to\infty}\frac{\forb(m,3,M)}{m2^{m-1}}-1=\lim_{m\to\infty}\frac{H(m,2)}{m2^{m-1}}=\lim_{m\to\infty}\frac{H_{2}(m,2)}{m2^{m-1}}=\lambda(2).$$
\end{conj}

The rest of the paper is organised as follows. In Section \ref{prelim}, we introduce two key concepts that will be used throughout this paper. In Section \ref{choice}, we show that to every extremal $A\in\Avoid(m,r,M)$, we can associate what is called a \emph{choice}, which is a collection $\mathcal{B}$ of $\binom{m}3$ $3\times3$ matrices, one for each triple of rows in $A$. Roughly speaking, for each triple of rows in $A$, the corresponding $3\times3$ matrix in $\mathcal{B}$ witnesses the fact that $A$ restricted to these three rows contains no configuration of $M$. We further distinguish every choice as either \emph{good} or \emph{bad}. In Section \ref{multigraph}, we encode the information contained in every choice $\mathcal{B}$ into a directed multigraph $\mathcal{D}_{\mathcal{B}}$, which allows us to analyse the choices and, in particular, compute $\forb(m,r,M)$ more easily. 

In Section \ref{proofof1.5}, we prove Theorem \ref{main}, which links $\forb(m,r,M)$ with the quantities $H(m,\frac{r-1}{r-2})$ and $H_2(m,\frac{r-1}{r-2})$ introduced in Question \ref{triangleproblem}. We first prove this in Section \ref{goodchoice} under the additional assumption that the choice $\mathcal{B}$ associated to the extremal matrix $A$ is good. To this end, we define an additional undirected multigraph $\mathcal{G}_{\mathcal{B}}$, and write $\forb(m,r,M)$ in terms of its edge multiplicities. Then, we finish the proof in Section \ref{reduction} by dealing with the case when the associated choice $\mathcal{B}$ is not necessarily good. 

In Section \ref{uppersection}, we prove Theorem \ref{thm:upperbds} and Theorem \ref{thm:generalupperbds}. First, in Section \ref{closedsets}, we define the concept of closed sets, which will help us analyse the structures of triangular choice multigraphs. Then, in Section \ref{atleast2}, we prove a key lemma showing that when $\alpha=2$, there exists an extremal TCM whose maximal closed sets all have sizes at least 2. In Section \ref{upperproof}, we use this lemma to prove Theorem \ref{thm:upperbds}, which gives an upper bound on $H(m,2)$, and then deduce Theorem \ref{thm:generalupperbds} from it. 

In Section \ref{lowersection}, we prove Theorem \ref{thm:generallowerbds} and Theorem \ref{thm:lowerbds}. We begin by using the definition of 2-recursive TCM to find a recurrence relation for $H_2(m,\alpha)$ that involves maximising over all possible bipartitions of $[m]$, which we call \emph{splits}. Then we prove Theorem \ref{thm:generallowerbds} in Section \ref{lowerboundconstruction}, which gives a lower bound on $H_2(m,\alpha)$, by specifying a particular split for every $m$. In Section \ref{h2exact}, we prove Theorem \ref{thm:lowerbds}, which completely determines the values of $H_{2}(m,2)$ by showing the splits we specified in Section \ref{lowerboundconstruction} are actually optimal.

In Section \ref{closingremarks}, we begin by stating generalisations to Theorem \ref{thm:upperbds} and Theorem \ref{thm:lowerbds} that we proved for the cases when $\alpha\geq 2$, and make an analogous conjecture to Conjecture \ref{conj}. Finally, we comment on this conjecture and the open cases of $r\geq 4$ and $\alpha<2$.

\section{Preliminaries}\label{prelim}
In this section, we introduce two key concepts that will be used throughout this paper. 
\subsection{Choice}\label{choice}
In this subsection, we associate what is called a choice to every extremal $A\in\Avoid(m,r,M)$. Let $A\in\Avoid(m, r, M)$ be extremal. Note that for each triple $i,j,k$ of rows, $A$ restricted to rows $i, j, k$ does not contain any configuration of $M$ is equivalent to $A$ restricted to rows $i, j, k$ containing at most one column in each of the following three pairs.
\begin{equation}\label{eq:misspairs}
\left\{\begin{bmatrix} 1\\ 0\\ 0 \end{bmatrix},  \begin{bmatrix} 0\\ 1\\ 1 \end{bmatrix}\right\},
\left\{\begin{bmatrix} 0\\ 1\\ 0 \end{bmatrix},  \begin{bmatrix} 1\\ 0\\ 1 \end{bmatrix}\right\},
\left\{\begin{bmatrix} 0\\ 0\\ 1 \end{bmatrix},  \begin{bmatrix} 1\\ 1\\ 0 \end{bmatrix}\right\}
\end{equation}

Since $A$ contains the maximum number of distinct columns, for each triple $i,j,k$ of rows, $A$ restricted to those rows must contain exactly one column in each of the three pairs above, as otherwise we can obtain a matrix in $\Avoid(m,r,M)$ that contains strictly more columns. This leads to the following definitions. 

\begin{definition}
A \textit{choice} on $V$ is a collection $\mathcal{B}=(B_{i,j,k}\colon\{i,j,k\}\in V^{(3)})$ of $3\times 3$ matrices, where each $B_{i,j,k}$ is formed by picking one column from each of the three pairs in (\ref{eq:misspairs}) and putting them together in that order. For every $X\subset V$, the \textit{restriction} of a choice $\mathcal{B}$ on $V$ to $X$ is the choice $\mathcal{B}|_X=(B_{i,j,k}\colon\{i,j,k\}\in X^{(3)})$ on $X$.

Let $A\in\Avoid(m,r,M)$. We say that rows $i,j,k$ \textit{forbid} a $3\times 3$ matrix $B_{i,j,k}$ if $A$ restricted to rows $i,j,k$ contains no column of $B_{i,j,k}$. If $A$ is extremal, then there is a unique $3\times3$ matrix $B_{i,j,k}$ forbidden by each triple $i,j,k$ of rows, hence every extremal $A$ is associated with a choice $\mathcal{B}=(B_{i,j,k}\colon\{i,j,k\}\in[m]^{(3)})$ on $[m]$.

We call a choice $\mathcal{B}$ \textit{good} if none of the $B_{i,j,k}$ is $I$ or $I^c$, and \textit{bad} otherwise. 

Given a choice $\mathcal{B}$ on $[m]$, define $\Avoid(m,r,\mathcal{B})$ to be the set of all simple $m$-rowed  matrices $A$ such that for each triple $i,j,k$, $A$ restricted to rows $i,j,k$ contains no columns of $B_{i,j,k}$. Also, define $\forb(m,r,\mathcal{B})=\max\{\abs{A}\colon A\in\Avoid(m,r,\mathcal{B})\}$.
\end{definition}

It is clear from the definition above that $\Avoid(m,r,M)$ is the union of $\Avoid(m,r,\mathcal{B})$ taken over all possible choices $\mathcal{B}$ on $[m]$. Hence, $$\forb(m,r,M)=\max\{\forb(m,r,\mathcal{B})\colon\mathcal{B}\text{ is a choice on }[m]\}.$$

\begin{definition}
For any column $c$, define its \emph{support} to be the index set of its $(0,1)$-entries. For a choice $\mathcal{B}$ on $V$ and any column $c$ with row index set $V$, we say $c$ is \textit{valid} with respect to $\mathcal{B}$ if for every triple $i,j,k$ of rows, $c$ restricted to these rows does not contain a column of $B_{i,j,k}$. 

Now let $\mathcal{B}$ be a choice on $[m]$ and let $X\subset[m]$. Define $\mathcal{C}(\mathcal{B},X)$ to be the set of valid $(0,1)$-columns with row set $X$ with respect to $\mathcal{B}|_X$, and define $c(\mathcal{B}, X)=\abs{\mathcal{C}(\mathcal{B}, X)}$. Similarly, let $\mathcal{C}^*(\mathcal{B}, X)$ be the set of all valid columns on $[m]$ with respect to $\mathcal{B}$ that have support $X$, and let $c^*(\mathcal{B}, X) = \abs{\mathcal{C}^*(\mathcal{B}, X)}$.
\end{definition}

Observe that $c^*(\mathcal{B}, X) = c(\mathcal{B}, X) (r-2)^{m-\abs{X}}$ since any of the $r-2$ symbols in $\{2,\cdots,r-1\}$ can be used on $[m] \setminus X$.
The following lemma shows that to compute $\forb(m,r,\mathcal{B})$, it suffices to count valid columns with different supports separately. 

\begin{lemma}\label{separatecount}
For any choice $\mathcal{B}$ on $[m]$, we have $$\forb(m,r,\mathcal{B})=\sum_{X\subset[m]}c^*(\mathcal{B}, X)=\sum_{X\subset[m]}c(\mathcal{B},X)(r-2)^{m-|X|}.$$
\end{lemma}
\begin{proof}
Let $Y=\bigcup_{X\subset[m]}\mathcal{C}^*(\mathcal{B}, X)$. Note that for any $X_1,X_2\subset[m]$ and $X_1\neq X_2$, $\mathcal{C}^*(\mathcal{B},X_1)\cap\mathcal{C}^*(\mathcal{B},X_2)=\emptyset$. Therefore, $|Y|=\sum_{X\subset[m]}c^*(\mathcal{B}, X)$. Consider the matrix $A$ consisting of all the columns in $Y$ exactly once. Then, $A\in\Avoid(m,r,\mathcal{B})$ and hence $\forb(m,r,\mathcal{B})\geq|A|= \sum_{X\subset[m]}c^*(\mathcal{B}, X)$.

On the other hand, let $A'\in \Avoid(m,r,\mathcal{B})$. For every column $c$ in $A'$, if the support of $c$ is $X_c \subset [m]$, then $c \in\mathcal{C}^*(\mathcal{B},X_c)\subset Y$. Hence, $|A'|\leq|Y|=\sum_{X\subset[m]}c^*(\mathcal{B}, X)$ and so $\forb(m,r,\mathcal{B}) \leq \sum_{X\subset[m]}c^*(\mathcal{B}, X)$. Therefore, $\forb(m,r,\mathcal{B}) = \sum_{X\subset[m]}c^*(\mathcal{B}, X)$.
\end{proof}

\subsection{Multigraph}\label{multigraph}
In this subsection, we show how to encode the forbidden conditions carried by a choice $\mathcal{B}$ into a directed multigraphs. 
\begin{definition}\label{def:0-imp}
Let $X\subseteq [m]$ and $\mathcal{B}$ be a choice on $[m]$. For distinct $i,j\in X$, we say that there is a \emph{$0$-implication} from $i$ to $j$ on $X$ if for any column $c$ with support $X$ in a matrix $A\in\Avoid(m,r,\mathcal{B})$, $c_i=0$ implies $c_j=0$. 
\end{definition}

Fix a choice $\mathcal{B}$ on $[m]$. It is easy to check that for each triple $i,j,k$ of rows, if $B_{i,j,k}\neq I, I^c$, then $B_{i,j,k}$ is a configuration of one of the following: 
\begin{equation*}
    A_1=\begin{bmatrix}0&0&0\\1&0&1\\1&1&0\end{bmatrix},\mbox{ 
    
   }A_2=\begin{bmatrix}0&0&1\\0&1&0\\1&1&1\end{bmatrix}.
\end{equation*}

Let $X\subset [m]$ and $i,j,k$ be a triple in $X$. Note that if $B_{i,j,k}=A_1$, then the restriction applied by forbidding $B_{i,j,k}$ is exactly equivalent to two $0$-implications on $X$, one from $i$ to $j$ and another from $i$ to $k$. Similarly, if $B_{i,j,k}=A_2$, then the restriction is equivalent to two $0$-implications on $X$, one from $i$ to $k$ and another from $j$ to $k$. In general, when $B_{i,j,k}$ is a configuration of $A_1$ or $A_2$, the restriction it applies is always equivalent to two $0$-implications on $X$ among $i,j,k$. Moreover, these two $0$-implications either both come from the same row index or both go to the same row index. 

We can now encode this information into a directed multigraph. 

\begin{definition}\label{directedmultigraph}
Given a choice $\mathcal{B}$ and $X\subseteq [m]$, define the directed multigraph $\mathcal{D}_{\mathcal{B}}(X)$ associated to $\mathcal{B}$ and $X$ as follows. For each triple $i,j,k$ of rows of $X$, if $B_{i,j,k}=I, I^c$, we do nothing. Otherwise, we draw two directed edges corresponding to the two $0$-implications equivalent to forbidding $B_{i,j,k}$, with an edge going from $x$ to $y$ if there is a $0$-implication from $x$ to $y$. For brevity, we write $\mathcal{D}_{\mathcal{B}}$ for $\mathcal{D}_{\mathcal{B}}([m])$.
\end{definition}
In general, the directed multigraph $\mathcal{D}_{\mathcal{B}}(X)$ is not simply the subgraph induced by $X$ in $\mathcal{D}_{\mathcal{B}}$. Indeed, if a directed edge $ij$ in $\mathcal{D}_{\mathcal{B}}$ comes from the triple $i,j,k$, and $X$ contains $i,j$ but not $k$, then at least this copy of the directed edge $ij$ is not in $\mathcal{D}_{\mathcal{B}}(X)$. However, $\mathcal{D}_{\mathcal{B}}(X)$ could still contain a directed edge $ij$ if $X$ contains another row $k'$ which contributes to a directed edge $ij$. From definition, we also see that $\mathcal{D}_{\mathcal{B}}$ is a directed multi-graph with $2\binom{m}{3}$ directed edges if $\mathcal{B}$ is a good choice. Figure \ref{fig:ExampleZeroImplications} shows an example directed multigraph associated to a good choice $\mathcal{B}$ on five vertices. In the next section, we will use the directed multigraphs $\mathcal{D}_{\mathcal{B}}(X)$ to help us compute the quantities $c(\mathcal{B},X)$ in Lemma \ref{separatecount}.

\begin{figure}[h]
\input{ExampleZeroImplications}
\end{figure}
\section{Proof of Theorem \ref{main}}\label{proofof1.5}
In this section, we prove Theorem \ref{main}, which relates  $\forb(m,r,M)$ to the quantities $H(m,\frac{r-1}{r-2})$ and $H_2(m,\frac{r-1}{r-2})$ introduced in Question \ref{triangleproblem}. We do this in two steps. Recall from Section \ref{choice} that $\forb(m,r,M)=\max\{\forb(m,r,\mathcal{B})\colon\mathcal{B}\text{ is a choice on }[m]\}$, and every choice is either good or bad. We first prove Theorem \ref{main} in Section \ref{goodchoice} assuming the choice $\mathcal{B}$ is good. Then, we finish the proof in Section \ref{reduction} by proving Theorem \ref{reduce}, which bounds the maximum of $\forb(m,r,\mathcal{B})$ over all bad choices $\mathcal{B}$.
\subsection{Good choice}\label{goodchoice}
In this subsection, we assume $\mathcal{B}$ is a good choice on $[m]$. To determine $\forb(m,r,\mathcal{B})$, it suffices to compute $c(\mathcal{B}, X)$ for each $X\subset[m]$ by Lemma \ref{separatecount}.

\begin{lemma}\label{boundonX}
Let $\mathcal{B}$ be a good choice on $[m]$ and let $X\subset[m]$. Then $c(\mathcal{B}, X)= n_t(\mathcal{B},X)+t(\mathcal{B},X)+1$, where $t(\mathcal{B},X)$ is the number of strongly connected components in $\mathcal{D}_{\mathcal{B}}(X)$ and $n_t(\mathcal{B},X)$ is the number of unordered pairs of strongly connected components in $\mathcal{D}_{\mathcal{B}}(X)$ such that there is no directed edge between them.

Consequently, $c(\mathcal{B}, X)\leq n(\mathcal{B},X)+\abs{X}+1$, where $n(\mathcal{B},X)$ is the number of unordered pairs of vertices in $X$ with no directed edge between them in $\mathcal{D}_{\mathcal{B}}(X)$. For simplicity, we sometimes say there is a non-edge between two such vertices. Equality is achieved if all strongly connected components in $\mathcal{D}_{\mathcal{B}}(X)$ are singletons. 
\end{lemma}
\begin{proof}
Fix a good choice $\mathcal{B}$ on $[m]$ and an $X\subset[m]$, write $t$ for $t(\mathcal{B},X)$ and let $C_1,\cdots,C_t$ be all the strongly connected components in $\mathcal{D}_{\mathcal{B}}(X)$. For any distinct $i,j\in[t]$, edges (if any) between vertices in $C_i$ and $C_j$ can only go in one direction, as otherwise $C_i\cup C_j$ is a larger strongly connected component. Hence, we may assume without loss of generality that $C_1,\cdots,C_t$ are ordered such that for any $i<j$ in $[t]$, edges (if any) between $C_i$ and $C_j$ always go from $C_i$ to $C_j$. 

Let $S=\{(i,j)\in[t]\times[t]\colon i<j\text{ and there is no directed edge between $C_i$ and $C_j$}\}$, then $|S|=n_t(\mathcal{B},X)$. We claim that $(i,j)\in S$ implies that $j=i+1$. Indeed, if $(i,j)\in S$ and $j>i+1$, let $x\in C_i, y\in C_{i+1}$ and $z\in C_j$. Then, as there is no directed edge between $C_i$ and $C_j$, the two directed edges corresponding to the two 0-implications created by $B_{x,y,z}$ must be $xy$ and $zy$, or $yx$ and $yz$. In the first case, the edge $zy$ goes from $C_j$ to $C_{i+1}$, while in the second case, the edge $yx$ goes from $C_{i+1}$ to $C_i$. Both contradict the ordering of the strongly connected components.

Now we compute $c(\mathcal{B},X)$, the number of valid (0,1)-columns $c$ with row set $X$ with respect to $\mathcal{B}|_X$. Such a column $c$ has to be constant on each strongly connected component. The column that is 1 on every component is a valid choice of $c$. For every $i\in[t]$, let $c$ be a valid column such that $C_i$ is the first component on which $c$ is 0. Then, $c$ is 1 on each $C_j$ with $j<i$ from assumption. Also, unless $(i,j)\in S$, $c$ is 0 on each $C_j$ with $j>i$ because there is a directed edge going from $C_i$ to $C_j$. In any case, the column $c$ that is 0 on $C_k$ for each $k\geq i$ and 1 on each $C_k$ with $k<i$ is always valid. As $i$ ranges in $[t]$, we get $t=t(\mathcal{B},X)$ such columns. From above, the only $j>i$ that could satisfy $(i,j)\in S$ is $j=i+1$. In this case, the column $c$ that is 0 on $C_k$ for each $k\geq i$ except $i+1$, and 1 on $C_{i+1}$ and each $C_k$ with $k<i$ is also valid. There are $|S|=n_t(\mathcal{B},X)$ such columns as $i$ ranges in $[t]$.

This implies that $c(\mathcal{B},X)=n_t(\mathcal{B},X)+t(\mathcal{B},X)+1$.
The second part of the lemma is is clear as $t(\mathcal{B},X)\leq\abs{X}$ and $n_t(\mathcal{B},X)\leq n(\mathcal{B},X)$ trivially. If all strongly connected components in $\mathcal{D}_{\mathcal{B}}(X)$ are singletons, then $t(\mathcal{B},X)=|X|$ and $n_t(\mathcal{B},X)=n(\mathcal{B},X)$, so equality is achieved. 
\end{proof}

Recall that $c^*(\mathcal{B},X) = c(\mathcal{B}, X) (r-2)^{m-\abs{X}}$. Lemma \ref{separatecount} and Lemma \ref{boundonX} imply that \begin{align}\label{eq:singlecomponents}
\forb(m,r,\mathcal{B})&=\sum_{X\subset[m]}c^*(\mathcal{B},X)\nonumber=\sum_{X\subset[m]} c(\mathcal{B}, X) (r-2)^{m-\abs{X}}\\
& \leq\sum_{X\subset[m]}(n(\mathcal{B},X)+\abs{X}+1)(r-2)^{m-\abs{X}} \nonumber \\
&= (r-1)^m + m(r-1)^{m-1} + \sum_{X\subset[m]}n(\mathcal{B},X)(r-2)^{m-\abs{X}}.
\end{align}  
To maximise the quantity $\sum_{X\subset[m]}n(\mathcal{B},X)(r-2)^{m-\abs{X}}$ over all good choices $\mathcal{B}$, we define an undirected multigraph $\mathcal{G}_{\mathcal{B}}$ that transforms this problem to Question \ref{triangleproblem}.

\begin{definition}\label{def:unordered}
Fix a good choice $\mathcal{B}$ on $[m]$. For every triple $x<y<z$, the restriction imposed by forbidding $B_{x,y,z}$ corresponds to 2 directed edges sharing a common vertex in the triangle $xyz$ in $\mathcal{D}_{\mathcal{B}}$, say $xy,xz$ or $yx,zx$. Then we add an undirected copy of the other edge, in this case $yz$, to the multigraph $\mathcal{G}_{\mathcal{B}}$. Moreover, for every unordered pair $xy\in[m]^{(2)}$, denote the multiplicity of edge $xy$ in $\mathcal{G}_{\mathcal{B}}$ by $m^{\mathcal{B}}_{xy}$, where we drop the superscript if the choice of $\mathcal{B}$ is clear. 
\end{definition}

Note that, by definition, $\mathcal{G}_{\mathcal{B}}$ is always a triangular choice multigraph and $m^{\mathcal{B}}_{xy}=m^{\mathcal{G}_{\mathcal{B}}}_{xy}$. On the other hand, for every triangular choice multigraph $\mathcal{G}$, we can find a good choice $\mathcal{B}$ such that $\mathcal{G}=\mathcal{G}_{\mathcal{B}}$. Figure \ref{fig:ExampleZeroImplications} provides an example of a directed multigraph $\mathcal{D}_{\mathcal{B}}$ associated to a good choice $\mathcal{B}$. The triangular choice multigraph $\mathcal{G}_{\mathcal{B}}$  associated to the same good choice $\mathcal{B}$ is given in Figure \ref{fig:ExampleTCM}.


The following lemma rewrites the sum in (\ref{eq:singlecomponents}) that we are trying to maximise, and brings it to the form of Question \ref{triangleproblem}.

\begin{lemma}\label{rewritesum}
For every good choice $\mathcal{B}$ on $[m]$, $$\sum_{X\subset[m]}n(\mathcal{B},X)(r-2)^{m-\abs{X}}= (r-2)^{m-2}\sum_{xy\in[m]^{(2)}}\left(\frac{r-1}{r-2}\right)^{m_{xy}}.$$
\end{lemma}
\begin{proof}

Let $xy \in [m]^{(2)}$ and let $\{x,y\}\subset X \subset [m]$ be such that there is no directed $xy$ or $yx$ edge in $\mathcal{D}_{\mathcal{B}}(X)$. Equivalently, for all vertices $z\in X\setminus\{x,y\}$, $xy$ is the edge among $xy,xz,yz$ added in $\mathcal{G}_{\mathcal{B}}$. Since there are $m_{xy}$ such $z\in[m]\setminus\{x,y\}$, there are $\binom{m_{xy}}{k}$ such choices of $X$ with $|X|=k+2$ for all $0\leq k\leq m_{xy}$, as we choose $k$ such $z$ alongside $x$ and $y$.

Therefore, we have
\begin{align*}
\sum_{X\subset[m]}n(\mathcal{B},X)(r-2)^{m-\abs{X}}&=\sum_{X\subset[m]} (r-2)^{m-|X|} \sum_{xy\in X^{(2)}}\mathbbm{1}(\text{there is no edge }xy\text{ in }\mathcal{D}_{\mathcal{B}}(X))\\
&=\sum_{xy\in [m]^{(2)}} \sum_{\{x,y\}\subset X\subset[m]} (r-2)^{m-|X|}\mathbbm{1}(\text{there is no edge }xy\text{ in }\mathcal{D}_{\mathcal{B}}(X))\\
&= \sum_{xy\in[m]^{(2)}} \sum_{k=0}^{m_{xy}}  \binom{m_{xy}}{k} (r-2)^{m-k-2}\\
&= \sum_{xy\in[m]^{(2)}} (r-2)^{m-2-m_{xy}} \sum_{k=0}^{m_{xy}}  \binom{m_{xy}}{k} (r-2)^{m_{xy} - k}\\
&= (r-2)^{m-2}\sum_{xy\in[m]^{(2)}} \left ( \frac{r-1}{r-2}\right )^{m_{xy}}, 
\end{align*}
as required. 
\end{proof}

We now prove Theorem \ref{main} assuming Theorem \ref{reduce}, which will be proved in the next section.
\begin{proof}[Proof of Theorem \ref{main}]
The cases when $m=0,1,2$ are trivial, so assume $m\geq 3$. Let $\mathcal{B}$ be a good choice on $[m]$ and let $\alpha_r = \frac{r-1}{r-2}$. Then, (\ref{eq:singlecomponents}) and Lemma \ref{rewritesum} imply
\begin{align*}
\forb(m,r,\mathcal{B})&\leq(r-1)^m+m(r-1)^{m-1}+(r-2)^{m-2} \sum_{xy\in[m]^{(2)}}\alpha_r^{m^{\mathcal{G}_{\mathcal{B}}}_{xy}}\\
&\leq(r-1)^m+m(r-1)^{m-1}+(r-2)^{m-2} H \left (m, \alpha_r \right),
\end{align*}
as $\mathcal{G}_{\mathcal{B}}$ is a triangular choice multigraph. Since this is true for every good choice $\mathcal{B}$, and Theorem \ref{reduce} shows $\forb(m,r,\mathcal{B})$ have the same upper bound for all bad choices $\mathcal{B}$, we have $\forb(m,r,M)\leq (r-1)^m+m(r-1)^{m-1}+H(m, \alpha_r)(r-2)^{m-2}$, as required.


On the other hand, let $\mathcal{G}$ be a 2-recursive triangular choice multigraph on $[m]$ such that $w(\mathcal{G},\alpha_r)=H_{2}(m, \alpha_r)$, and let $[m]=V_1\cup V_2$ and $\mathcal{G}_1, \mathcal{G}_2,\mathcal{G}_3$ be as in Definition \ref{tcmdef}. As permuting the vertex labels does not change $w(\mathcal{G},\alpha_r)$, we may assume there exists some $m'\in[m-1]$ such that $V_1=[m']$ and $V_2=[m]\setminus[m']$. Note that from the definition of 2-recursive TCM, we have $$w(\mathcal{G},\alpha_r)=w(\mathcal{G}_1,\alpha_r)\alpha_r^{m-m'}+w(\mathcal{G}_2,\alpha_r)\alpha_r^{m'}+m'(m-m').$$ Hence, $w(\mathcal{G},\alpha_r)=H_{2}(m, \alpha_r)$ implies that $w(\mathcal{G}_1,\alpha_r)=H_{2}(m', \alpha_r)$ and $w(\mathcal{G}_2,\alpha_r)=H_{2}(m-m', \alpha_r)$. 

We say a choice $\mathcal{B}$ on $[m]$ is \textit{uniformly directed} if for all $i<j$, all directed edges between $i$ and $j$ in $\mathcal{D}_{\mathcal{B}}$ point from $i$ to $j$. We use induction to show that for every $m\geq1$ and every 2-recursive TCM $\mathcal{G}$ on $[m]$ satisfying $w(\mathcal{G},\alpha_r)=H_2(m,\alpha_r)$, there exists a good choice $\mathcal{B}$ on $[m]$ such that $\mathcal{G}_\mathcal{B}=\mathcal{G}$ and ${\mathcal{B}}$ is uniformly directed. Note that empty choice works for $m=1,2$. For $m\geq 3$, by induction hypothesis, there exist uniformly directed good choices $\mathcal{B}_1$ and $\mathcal{B}_2$ on $[m']$ and $[m]\setminus[m']$, respectively, such that $\mathcal{G}_{\mathcal{B}_1}=\mathcal{G}_1$ and $\mathcal{G}_{\mathcal{B}_2}=\mathcal{G}_2$. For each triple $1\leq i<j\leq m'<k\leq m$, define the choice on $i,j,k$ so that the corresponding edges in $\mathcal{D}_{\mathcal{B}}$ are $ik$ and $jk$. For each triple $1\leq i\leq m'<j<k\leq m$, define the choice on $i,j,k$ so that the corresponding edges in $\mathcal{D}_{\mathcal{B}}$ are $ij$ and $ik$. Taking these together with $\mathcal{B}_1$ and $\mathcal{B}_2$ gives a uniformly directed good choice $\mathcal{B}$ satisfying $\mathcal{G}_{\mathcal{B}}=\mathcal{G}$, as required. In particular, for all $X\subset[m]$, strongly connected components in $\mathcal{D}_{\mathcal{B}}(X)$ are all singletons. Hence, by Lemma \ref{separatecount}, Lemma \ref{boundonX}, and Lemma \ref{rewritesum}, we have
\begin{align*}
\forb(m,r,M)&\geq\forb(m,r,\mathcal{B})=\sum_{X\subset[m]}c^*(\mathcal{B},X)\\
&=\sum_{X\subset[m]}(n(\mathcal{B},X)+\abs{X}+1)(r-2)^{m-\abs{X}}\\
&=(r-1)^m+m(r-1)^{m-1}+\sum_{X\subset[m]}n(\mathcal{B},X)(r-2)^{m-\abs{X}}\\
&=(r-1)^m+m(r-1)^{m-1}+(r-2)^{m-2}\sum_{xy\in[m]^{(2)}}\alpha_r^{m^{\mathcal{G}_{\mathcal{B}}}_{xy}}\\
&=(r-1)^m+m(r-1)^{m-1}+(r-2)^{m-2}H_{2}(m, \alpha_r),
\end{align*}
as required.
\end{proof}



\subsection{Bad choice}\label{reduction}
In this section we prove the following result, which completes the proof of Theorem \ref{main} in the previous subsection.
\begin{thm}\label{reduce}
For every $m\geq0$, $r\geq3$,
$$\max\{\forb(m,r,\mathcal{B})\colon \mathcal{B}\text{ is a bad choice}\}-(r-1)^m-m(r-1)^{m-1}\leq H\left(m,\frac{r-1}{r-2}\right)\cdot(r-2)^{m-2}.$$
\end{thm}

Let us fix a (possibly bad) choice $\mathcal{B}$ on $[m]$ and a subset $X\subset[m]$. Our goal is to bound $c(\mathcal{B},X)$ in a different way than Lemma \ref{boundonX}. To do this, we first modify the directed multigraph $\mathcal{D}_{\mathcal{B}}(X)$ by adding some more edges according to the following rules whenever possible.  
\begin{enumerate}
    \item If $i,j,k\in X$ are distinct and $ij$,  $jk$ are directed edges, then add directed edge $ik$. 
    \item If $i,j,k\in X$ are distinct, $B_{i,j,k}=I$ and $ij$ is a directed edge, then add directed edge $ik$.
    \item If $i,j,k\in X$ are distinct, $B_{i,j,k}=I^c$ and $ij$ is a directed edge, then add directed edge $kj$.
\end{enumerate}
Denote this new directed multigraph by $\overline{\mathcal{D}}_{\mathcal{B}}(X)$. It is easy to see that each directed edge in $\overline{\mathcal{D}}_{\mathcal{B}}(X)$ is indeed a $0$-implication on $X$ according to Definition \ref{def:0-imp}. Also note that Rule 1 above guarantees that $\overline{\mathcal{D}}_{\mathcal{B}}(X)$ is transitive.

For not necessarily distinct $i,j\in X$, we say they are \textit{independent}, and denote it as $i\sim j$, if there exists no directed edge between them in $\overline{\mathcal{D}}_{\mathcal{B}}(X)$.

\begin{lemma}\label{prop:nonedgeblocks}
The relation $\sim$ is an equivalence relation on $X$.
\end{lemma}
\begin{proof}
$\sim$ is trivially symmetric. $\sim$ is reflexive as the rules above never result in loops. Assume now that $i\sim j$ and $j\sim k$ for distinct $i,j,k\in X$, which implies that $B_{i,j,k}=I\text{ or }I^c$ in $\mathcal{B}$. Suppose $i\not\sim k$. Without loss of generality assume $ik$ is a directed edge in $\overline{\mathcal{D}}_{\mathcal{B}}(X)$. If $B_{i,j,k}=I$, then Rule 2 above implies the existence of a directed edge $ij$, contradicting $i\sim j$. If $B_{i,j,k}=I^c$, then Rule 3 above implies the existence of a directed edge $jk$, contradicting $j\sim k$.
\end{proof}
Call any equivalence class in $X$ under relation $\sim$ with respect to $\mathcal{B}$ a \emph{block}. Observe that if $i$ and $j$ are in different blocks, then there is a directed edge between them in $\overline{\mathcal{D}}_{\mathcal{B}}(X)$. 
\begin{lemma}\label{prop:samedirection}
Let $A, B$ be two blocks in $X$. If at least one of $A,B$ has size at least 2, then edges between $A$ and $B$ all go in the same direction.
\end{lemma}
\begin{proof}
Suppose not all edges between blocks $A$ and $B$ go in the same direction, then we can find 3 vertices, without loss of generality $a_1,a_2\in A$ and $b\in B$, such that we have directed edges $a_1b$ and $ba_2$. By the transitivity rule, we also have directed edge $a_1,a_2$, contradicting $A$ is a block. 
\end{proof}

Using a well-known result on tournaments, this implies that we can order the blocks in $\overline{\mathcal{D}}_{\mathcal{B}}(X)$ as $D_1,D_2,\cdots ,D_t$, so that for all $\gamma<\gamma'$ in $[t]$ and all $i\in D_\gamma,j\in D_\gamma'$, there is a directed $ij$ edge in $\overline{\mathcal{D}}_{\mathcal{B}}(X)$. If $D_\gamma,D_{\gamma'}$ are singleton blocks, then there might be a directed $ji$ edge as well, but we will not use them in the proofs below.

To get an upper bound on $c(\mathcal{B},X)$, the following observation will be useful.
\begin{lemma}\label{prop:IorIcbetween blocks}
Let $D$ be a block and let $i,j\in D$ be distinct. Furthermore suppose $k$ is in a block before $D$ in the ordering defined above. Then $B_{k,i,j}\neq I^c$. Similarly, if $k'$ is in a block after $D$, then $B_{i,j,k'}\neq I$.
\end{lemma}
\begin{proof}
If $k$ is in a block before $D$, then we have directed $ki$ and $kj$ edges. If $B_{k,i,j}=I^c$, then by Rule 3 we also have directed edges $ij,ji$, contradicting $i\sim j$. Similarly, if $k'$ is in a block after $D$, then we have directed $ik'$ and $jk'$ edges. If $B_{i,j,k'}=I$, then by Rule 2 we also have directed edges $ij,ji$, again contradicting $i\sim j$. 
\end{proof}

We can now obtain a preliminary upper bound on $c(\mathcal{B},X)$. Let $b(\mathcal{B},X)$ denote the number of blocks of size at least 2 on $X$ with respect to the choice $\mathcal{B}$.
\begin{lemma}\label{lem:blockbound}
For every choice $\mathcal{B}$ on $[m]$ and $X\subset[m]$, we have 
\begin{equation}\label{eq:blockbd}
 c(\mathcal{B},X)\le |X|+b(\mathcal{B},X)+1.   
\end{equation}
\end{lemma}
\begin{proof}
Let the blocks in $X$ be ordered as $D_1,D_2,\cdots,D_t$ as before. For each $\gamma\in[t]$, let $C(\mathcal{B},X,\gamma)$ be the set of columns in $C(\mathcal{B},X)$ that have their first 0 in block $D_{\gamma}$. We claim that $|C(\mathcal{B},X,\gamma)|\leq|D_\gamma|+1$ if $|D_\gamma|\geq 2$. Indeed, every column $c\in C(\mathcal{B},X,\gamma)$ contains only 1's in each block $D_\beta$ for $1\le \beta<\gamma$ and only 0's in each block $D_{\beta'}$ for $\gamma<\beta'\le t$. If $|D_{\gamma}|=1$, then there is at most $1=|D_{\gamma}|$ such column. If $|D_{\gamma}|=2$, then there are at most $3=|D_{\gamma}|+1$ such columns, namely those that are $00,10$ or $01$ on $D_\gamma$. 

Now assume $|D_{\gamma}|\geq3$. Since $D_\gamma$ is a block, for every distinct $i,j,k\in D_{\gamma}$ we have $B_{i,j,k}=I\text{ or }I^c$.

\mbox{}

\noindent\textbf{Case 1.} There is a column $c\in C(\mathcal{B},X,\gamma)$ that contains at least two 0's and at least two 1's on $D_{\gamma}$. Let $L\subset D_{\gamma}$ be the set of row indices on which $c$ is 0, while $R\subset D_{\gamma}$ be the set of indices on which $c$ is 1. Then for all distinct $i,j,k\in D_\gamma$, we have
\begin{enumerate}[(i)]
    \item if $|\{i,j,k\}\cap L|=2$ then $B_{i,j,k}=I^c$,
    \item if $|\{i,j,k\}\cap R|=2$ then $B_{i,j,k}=I$. 
\end{enumerate}
We claim that in this case every column $c'\in C(\mathcal{B},X,\gamma)$ is constant on $L$ and on $R$. Indeed, if there exist $i,j\in L$ such that $c'_i=0,c'_j=1$, then $c'$ is constantly 0 on $R$ by (i). But then for any $k,k'\in R$, $c'_j=1,c'_k=c'_{k'}=0$, contradicting (ii). Similarly, if there exist $j,k\in R$ such that $c'_j=0,c'_k=1$, then $c'$ is constantly 1 on $L$ by (ii). But then for any $i,i'\in L$, $c'_j=0,c'_i=c'_{i'}=1$, contradicting (i).  Hence, there are at most $3\leq|D_\gamma|+1$ columns in $C(\mathcal{B},X,\gamma)$, namely the one that is 0 on both $L$ and $R$, the one that is 0 on $L$, 1 on $R$, and the one that is 1 on $L$ and 0 on $R$.

\mbox{}

\noindent\textbf{Case 2.} No column in $C(\mathcal{B},X,\gamma)$ contains at least two 0's and two 1's on $D_{\gamma}$. We claim that $|C(\mathcal{B},X,\gamma)|\leq|D_{\gamma}|+1$. This is true if all columns in $C(\mathcal{B},X,\gamma)$ contain at least two 0's on $D_\gamma$, as then all of them contain at most one 1 on $D_\gamma$. Otherwise, there is a column $c\in C(\mathcal{B},X,\gamma)$ that has only one 0, say at $i\in D_\gamma$. For all other distinct $j,k\in D_\gamma$, $c_j=c_k=1$, so $B_{i,j,k}=I$. Let $c'$ be any other column in $C(\mathcal{B},X,\gamma)$. If $c'$ has at least two 0's and a 1 on $D_{\gamma}$, then it has exactly one 1 on $D_{\gamma}$, say at $i'$. If $i'=i$, then for any other distinct $j,k\in D_\gamma$, $c'_i=1,c'_j=c'_k=0$, contradicting $B_{i,j,k}=I$. If $i'\not=i$, then for any other $j\in D_\gamma$, $c'_i=c'_j=0,c'_{i'}=1$, again contradicting $B_{i,i',j}=I$. Hence, any such $c'$ is either all 0 or has one $0$, so there are at most $|D_\gamma|+1$ such columns in this case as well. 

\mbox{}

Putting all these together, we have
\[
 c(\mathcal{B},X)=1+\sum_{\gamma=1}^t|C(\mathcal{B},X,\gamma)|\leq 1+\sum_{\gamma=1}^t|D_{\gamma}|+b(\mathcal{B},X)=|X|+b(\mathcal{B},X)+1,\]
where the sum counts all the columns with at least one 0 on $X$, and the extra 1 is the column with all 1's on $X$.
\end{proof}

We need a strengthening of bound (\ref{eq:blockbd}). Let us call a block $D_{\gamma}$ \emph{special} if there exist distinct $i,j\in D_\gamma$ such that either there exists some $k$ in a block before $D_\gamma$ satisfying $B_{k,i,j}=I$, or there exists some $k$ in a block after $D_\gamma$ such that $B_{i,j,k}=I^c$. Let $s(\mathcal{B},X)$ be the number of special blocks.
\begin{prop}\label{prop:specblock}
For every choice $\mathcal{B}$ on $[m]$ and $X\subset[m]$, we have
\begin{equation*}
 c(\mathcal{B},X)\le |X|+b(\mathcal{B},X)-s(\mathcal{B},X)+1.   
\end{equation*}
\end{prop}
\begin{proof}
Order the blocks $D_1,D_2,\cdots,D_t$ in $X$ as before. Note that by definition every special block has size at least 2. Let $S$ be the set of indices $\gamma\in[t]$ such that $D_\gamma$ is a special block. Let $S'\subset S$ be the set of indices $\gamma$ such that $D_\gamma$ is a special block of size at least 3 that falls under Case 1 in Lemma \ref{lem:blockbound}; in other words, $\gamma \in S'$ if $C(\mathcal{B},X,\gamma)$ contains a column with at least two 0's and at least two 1's on $D_\gamma$.

For each special block $D_\gamma$ satisfying $\gamma\in S'$, we actually showed in the proof of Lemma \ref{lem:blockbound} that the number of columns in $C(\mathcal{B},X,\gamma)$ is at most $3\leq|D_\gamma|<|D_\gamma|+1$. This improves the bound on $c(\mathcal{B},X)$ in Lemma \ref{lem:blockbound} by at least $|S'|$.

Now for each $\gamma\in S\setminus S'$, we will find a column $c^\gamma$ that is counted in the upper bound proof of Lemma \ref{lem:blockbound}, but is actually not in $C(\mathcal{B},X)$. Note that we will not require that $c^\gamma\in C(\mathcal{B},X,\gamma)$, meaning $c^\gamma$ need not have its first 0 in $D_\gamma$, but we will make sure that these $c^\gamma$ are distinct for different $\gamma\in S\setminus S'$. Therefore, the collection $\{c^\gamma\colon\gamma\in S\setminus S'\}$ of columns together will improve the bound on $c(\mathcal{B},X)$ in Lemma \ref{lem:blockbound} by $|S\setminus S'|$. We find $c^\gamma$ as follows. Since $D_\gamma$ is special, we can find $i,j\in D_\gamma$, and either some $k$ in a block before $D_\gamma$ satisfying $B_{k,i,j}=I$ or some $k$ in a block after $D_\gamma$ with $B_{i,j,k}=I^c$.

If $k$ is in some block before $D_\gamma$, let $c^\gamma$ be the column given by the following table. Note that this column $c^\gamma$ is counted in the upper bound proof of Lemma \ref{lem:blockbound}, but is not in $C(\mathcal{B},X)$ because $B_{k,i,j}=I$.
\[
\begin{tabular}{|c||c|c|}
\hline\multirow{2}{1em}{\centering$x$} & $x\in D_\beta, \beta<\gamma$ & $x\in D_{\beta'},\beta'\geq\gamma$\\
& (including $x=k$) & (including $x=i,j$) \\\hline
$c^\gamma_x$ & 1 & 0\\\hline
\end{tabular}
\]
Suppose instead $k$ is in some block after $D_{\gamma}$. Let $c^\gamma$ be the column given by the following table. Again, $c^\gamma$ is counted in the upper bound proof of Lemma \ref{lem:blockbound}, but is not in $C(\mathcal{B},X)$ because $B_{i,j,k}=I^c$.
\[
\begin{tabular}{|c||c|c|}
\hline\multirow{2}{1em}{\centering$x$} & $x\in D_\beta, \beta\leq\gamma$ & $x\in D_{\beta'},\beta'>\gamma$\\
& (including $x=i,j$) & (including $x=k$) \\\hline
$c^\gamma_x$ & 1 & 0\\\hline
\end{tabular}
\]
Note that for $\gamma'<\gamma$ in $S\setminus S'$, columns $c^{\gamma'}$ and $c^\gamma$ obtained in this way are distinct, unless $\gamma'=\gamma-1$ and there exist $i,j\in D_{\gamma-1}$, $i',j'\in D_\gamma$, $k$ in some block after $D_{\gamma-1}$ (possibly equal to $i'$ or $j'$) and $k'$ in some block before $D_\gamma$ (possibly equal to $i$ or $j$), such that $D_{\gamma-1}$ is special because $B_{i,j,k}=I^c$ and $D_\gamma$ is special because $B_{k',i',j'}=I$. In what follows, for each pair $\gamma-1,\gamma\in S\setminus S'$ satisfying $c^{\gamma-1}=c^\gamma$, we replace $c^\gamma$ with a different column that is also counted in the upper bound proof of Lemma \ref{lem:blockbound} but not in $C(\mathcal{B},X)$, so that the new collection $\{c^\gamma\colon\gamma\in S\setminus S'\}$ is actually a set of $|S\setminus S'|$ pairwise distinct columns. We need to separate into several cases.

\mbox{}

\noindent\textbf{Case 1.} $k$ is in a block after $D_\gamma$. In this case, consider the two columns defined below. 
\[
\begin{tabular}{|c||c|c|c|c|}
\hline\multirow{2}{1em}{\centering$x$} & $x\in D_\beta, \beta\leq\gamma-1$ & $x\in D_\gamma, x\not=j'$ & \multirow{2}{3em}{\centering$x=j'$}& $x\in D_{\beta'},\beta'\geq\gamma+1$\\
& (including $x=i,j,k'$) & (including $x=i'$) & & (including $x=k$) \\\hline
$\overline{c}^\gamma_x$ & 1 & 1 & 0 & 0\\\hline
\end{tabular}
\]
\[
\begin{tabular}{|c||c|c|c|c|}
\hline\multirow{2}{1em}{\centering$x$} & $x\in D_\beta, \beta\leq\gamma-1$ & $x\in D_\gamma, x\not=j'$ & \multirow{2}{3em}{\centering$x=j'$}& $x\in D_{\beta'},\beta'\geq\gamma+1$\\
& (including $x=i,j,k'$) & (including $x=i'$) & & (including $x=k$) \\\hline
$\tilde{c}^\gamma_x$ & 1 & 0 & 1 & 0\\\hline
\end{tabular}
\]
Neither $\overline{c}^\gamma_x$ nor $\tilde{c}^\gamma_x$ is in $C(\mathcal{B},X)$ as $B_{i,j,k}=I^c$. Both $\overline{c}^{\gamma}_x$ and $\tilde{c}^{\gamma}_x$ have their first 0 in block $D_{\gamma}$, where $\overline{c}^{\gamma}$ is all 1 except $\overline{c}^{\gamma}_{j'}=1$, while $\tilde{c}^{\gamma}$ is all 0 except $\tilde{c}^{\gamma}_{j'}=1$. Hence, one of $\overline{c}^{\gamma},\tilde{c}^{\gamma}$ is counted in the upper bound proof of Lemma \ref{lem:blockbound}. Replace $c^{\gamma}$ with this one.

\mbox{}

\noindent\textbf{Case 2.} $k\in D_{\gamma}$, $k\not=i',j'$. Once again, consider the two columns defined below. 
\[
\begin{tabular}{|c||c|c|c|c|}
\hline\multirow{2}{1em}{\centering$x$} & $x\in D_\beta, \beta\leq\gamma-1$ & $x\in D_{\gamma}, x\not=k$ & \multirow{2}{3em}{\centering$x=k$} & \multirow{2}{8.5em}{\centering$x\in D_{\beta'}, \beta'\geq\gamma+1$}\\
& (including $x=i,j,k'$) & (including $x=i',j'$) & & \\\hline
$\overline{c}^{\gamma}_x$ & 1 & 1 & 0 & 0\\\hline
\end{tabular}
\]
\[
\begin{tabular}{|c||c|c|c|c|}
\hline\multirow{2}{1em}{\centering$x$} & $x\in D_\beta, \beta\leq\gamma-1$ & $x\in D_{\gamma}, x\not=k$ & \multirow{2}{3em}{\centering$x=k$}& \multirow{2}{8.5em}{\centering$x\in D_{\beta'},\beta'\geq\gamma+1$}\\
& (including $x=i,j,k'$) & (including $x=i',j'$) & & \\\hline
$\tilde{c}^{\gamma}_x$ & 1 & 0 & 1 & 0\\\hline
\end{tabular}
\]
Note that $\overline{c}^{\gamma}_x$ is not in $C(\mathcal{B},X)$ as $B_{i,j,k}=I^c$ and $\tilde{c}^{\gamma}_x$ is not in $C(\mathcal{B},X)$ as $B_{k',i',j'}=I$. Both $\overline{c}^{\gamma}_x$ and $\tilde{c}^{\gamma}_x$ have their first 0 in block $D_{\gamma}$, where $\overline{c}^{\gamma}$ is all 1 except $\overline{c}^{\gamma}_{k}=0$, while $\tilde{c}^{\gamma}$ is all 0 except $\tilde{c}^{\gamma}_{k}=1$. Hence, one of $\overline{c}^{\gamma},\tilde{c}^{\gamma}$ is counted in the upper bound proof of Lemma \ref{lem:blockbound}. Replace $c^{\gamma}$ with this one.

\mbox{}

\noindent\textbf{Case 3.} $k\in D_{\gamma}$, $k=i'$ or $j'$, and $|D_{\gamma}|=3$. Without loss of generality, say $k=j'$, and let $\ell$ be another index in $D_{\gamma}$ not equal to $i',j',k$. Again, consider the two columns defined below. 
\[
\begin{tabular}{|c||c|c|c|c|}
\hline\multirow{2}{1em}{\centering$x$} & $x\in D_\beta, \beta\leq\gamma-1$ & $x\in D_{\gamma}, x\not=j',k$ & \multirow{2}{5em}{\centering$x=j'=k$} & \multirow{2}{8.5em}{\centering$x\in D_{\beta'}, \beta'\geq\gamma+1$}\\
& (including $x=i,j,k'$) & (including $x=i',\ell$) & & \\\hline
$\overline{c}^{\gamma}_x$ & 1 & 1 & 0 & 0\\\hline
\end{tabular}
\]
\[
\begin{tabular}{|c||c|c|c|c|}
\hline\multirow{2}{1em}{\centering$x$} & $x\in D_\beta, \beta\leq\gamma-1$ & $x\in D_{\gamma}, x\not=\ell$ & \multirow{2}{3em}{\centering$x=\ell$}& \multirow{2}{8.5em}{\centering$x\in D_{\beta'},\beta'\geq\gamma+1$}\\
& (including $x=i,j,k'$) & (including $x=i',x=j'=k$) & & \\\hline
$\tilde{c}^{\gamma}_x$ & 1 & 0 & 1 & 0\\\hline
\end{tabular}
\]
Neither $\overline{c}^{\gamma}_x$ nor $\tilde{c}^{\gamma}_x$ is in $C(\mathcal{B},X)$ as $B_{i,j,k}=I^c$. Both $\overline{c}^{\gamma}_x$ and $\tilde{c}^{\gamma}_x$ have their first 0 in block $D_{\gamma}$, where $\overline{c}^{\gamma}$ is all 1 except $\overline{c}^{\gamma}_{k}=1$, while $\tilde{c}^{\gamma}$ is all 0 except $\tilde{c}^{\gamma}_{\ell}=1$. Hence, one of $\overline{c}^{\gamma},\tilde{c}^{\gamma}$ is counted in the upper bound proof of Lemma \ref{lem:blockbound}. Replace $c^{\gamma}$ with this one.

\mbox{}

\noindent\textbf{Case 4.} $k\in D_{\gamma}$, $k=i'$ or $j'$, and $|D_{\gamma}|=2$. Without loss of generality
, assume $k=j'$. Consider the column defined below. 
\[
\begin{tabular}{|c||c|c|c|c|}
\hline\multirow{2}{1em}{\centering$x$} & $x\in D_\beta, \beta\leq\gamma-1$ & \multirow{2}{3em}{\centering$x=i'$} & \multirow{2}{5em}{\centering$x=j'=k$} & \multirow{2}{8.5em}{\centering$x\in D_{\beta'}, \beta'\geq\gamma+1$}\\
& (including $x=i,j,k'$) & & & \\\hline
$\overline{c}^{\gamma}_x$ & 1 & 1 & 0 & 0\\\hline
\end{tabular}
\]
Note that $\overline{c}^{\gamma}_x$ is not in $C(\mathcal{B},X)$ as $B_{i,j,k}=I^c$. Since $|D_{\gamma}|=2$, $\overline{c}^{\gamma}_x$ is counted in the upper bound proof of Lemma \ref{lem:blockbound}. Replace $c^{\gamma}$ with this one.

We claim that after these replacements, all columns $c^\gamma$ for $\gamma\in S\setminus S'$ are pairwise distinct. Indeed, the columns $c^\gamma$ that were unchanged are pairwise distinct, as otherwise one of them would have been replaced. Moreover, all of these unchanged columns are constant on each block. For each column $c^\gamma$ that was updated, it has a unique non-constant block $D_\gamma$. Hence, the updated columns are pairwise distinct and also distinct from all the unchanged ones, proving the claim. Thus, $\{c^\gamma\colon\gamma\in S\setminus S'\}$ is now a set of $|S\setminus S'|$ columns that we have counted in the proof of Lemma \ref{lem:blockbound} that are not actually in $C(\mathcal{B},X)$.

Hence, we have
\begin{align*}
 c(\mathcal{B},X)&=1+\sum_{\gamma=1}^t|C(\mathcal{B},X,\gamma)|\\
 &\le 1+\sum_{\gamma\in[t]\setminus S}|C(\mathcal{B},X,\gamma)|+\sum_{\gamma\in S'}|D_{\gamma}|+\left(\sum_{\gamma\in S\setminus S'}(|D_\gamma|+1)-|S\setminus S'|\right)\\
 &=|X|+b(\mathcal{B},X)+1-|S'|-|S\setminus S'|=|X|+b(\mathcal{B},X)-s(\mathcal{B},X)+1,
\end{align*}
as required.
\end{proof}

We are now ready to prove Theorem \ref{reduce}, which completes the proof of Theorem \ref{main}.
\begin{proof}[Proof of Theorem \ref{reduce}]
Let $\mathcal{B}$ be any bad choice on $[m]$. It suffices to find a good choice $\mathcal{B}'$ on $[m]$ such that for all $X\subset[m]$,
\begin{equation*}
|X|+b(\mathcal{B},X)-s(\mathcal{B},X)+1\le n(\mathcal{B}',X)+|X|+1,
\end{equation*}
where $n(\mathcal{B}',X)$ is calculated as in Lemma \ref{boundonX}. We will do this by replacing each $B_{i,j,k}\in\mathcal{B}$ that is equal to $I$ or $I^c$ with a configuration of $A_2$ to obtain a good choice $\mathcal{B}'$, such that for any $X\subset[m]$, every non-special block on $X$ of size at least 2 contains a non-edge of $\mathcal{D}_{\mathcal{B'}}(X)$. Since there are $b(\mathcal{B},X)-s(\mathcal{B},X)$ non-special block of size at least 2, we get $n(\mathcal{B}',X)\geq b(\mathcal{B},X)-s(\mathcal{B},X)$, as required. 

Order the blocks determined by $\mathcal{B}$ on $[m]$ as usual and call them \emph{original blocks}.
Furthermore, fix an arbitrary ordering of the elements inside each original block so that together with the ordering of the blocks, we obtain a total order on $[m]$. Denote this order by $\prec$. For all triples $i\prec j\prec k$ in $[m]$ such that $B_{i,j,k}=I$ or $I^c$, we replace $B_{i,j,k}$ with $B'_{i,j,k}$, which is a configuration of $A_2$ that corresponds to 0-implications from $i$ to $k$ and from $j$ to $k$. We claim that the choice $\mathcal{B}'$ obtained in this way works. 

Indeed, let $X\subset[m]$ and let $D$ be a non-special block of size at least 2 on $X$ induced by $\mathcal{B}$. First suppose $D=\{i,j\}$ has size 2. Since $D$ is a block, there is no directed $ij$ or $ji$ edge in $\overline{\mathcal{D}}_{\mathcal{B}}(X)$. As every directed edge in $\mathcal{D}_{\mathcal{B}}(X)$ is in $\overline{\mathcal{D}}_{\mathcal{B}}(X)$, there is also no directed $ij$ or $ji$ edge in $\mathcal{D}_{\mathcal{B}}(X)$. Moreover, for any $k$ not in $D$, as $D$ is a non-special block, we have that $B_{i,j,k}$ is a configuration of $A_1$ or $A_2$ that corresponds to no directed $ij$ or $ji$ edge. This means that $B_{i,j,k}=B'_{i,j,k}$ is unchanged for every $k$. Thus, the replacements do not introduce any new directed $ij$ or $ji$ edge, so $ij$ is a non-edge of $\mathcal{D}_{\mathcal{B'}}(X)$ as required.

On the other hand, if $|D|\ge 3$, let $i\prec j$ be the two smallest elements in $D$ with respect to the order $\prec$. As above, there is no directed $ij$ or $ji$ edge in $\mathcal{D}_{\mathcal{B}}(X)$, and $B_{i,j,k}=B'_{i,j,k}$ for all $k\not\in D$. Suppose that $B_{i,j,k}=I\text{ or }I^c$ for some $k\in D$, then $i\prec j\prec k$ and by construction, replacing $B_{i,j,k}$ with $B_{i,j,k}'$ adds directed $ik$ and $jk$ edges to $\mathcal{D}_{\mathcal{B}}(X)$, but not any directed $ij$ or $ji$ edge. Since this is true for all such $k\in D$, $ij$ forms a non-edge in $\mathcal{D}_{\mathcal{B'}}(X)$, as required.
\end{proof}

\section{Upper bounds}\label{uppersection}
In this section, we focus on upper bounds. For simplicity, we write $H(m)=H(m,2)$ and $w(\mathcal{G})=w(\mathcal{G},2)$ throughout this section. We first obtain an upper bound for $H(m)$ by proving Theorem \ref{thm:upperbds}. The values of $H(m)$ and $h(m)=H(m)/m2^{m-1}$ (rounded to 3 decimals) for $1\leq m\leq6$ are given in the table below. 
\begin{table}[h]
\centering
\begin{tabular}{ |c|c|c|c|c|c|c| }
 \hline
 $m$ & 1 & 2 & 3 & 4 & 5 & 6\\
 \hline
 $H(m)$ & 0 & 1 & 4 & 12 & 30 & 73\\
 \hline
 $h(m)$ & 0.000 & 0.250 & 0.333 & 0.375 & 0.375 & 0.380\\
 \hline
\end{tabular}
\end{table}
As the number of TCM on $[m]$ grows exponentially in $m$, it quickly becomes infeasible to calculate $H(m)$ and $h(m)$ by brute force. Instead, we introduce the concept of closed sets in Section \ref{closedsets}, which helps us decompose and better analyse triangular choice multigraphs. In Section \ref{atleast2}, we prove a key lemma which shows that there exists an extremal TCM such that all of its maximal closed sets have size at least 2. Then, in Section \ref{upperproof}, we use this lemma to prove Theorem \ref{thm:upperbds}, and deduce from it Theorem \ref{thm:generalupperbds}, which gives an upper bound on $\forb(m,r,M)$ for all $r\geq 3$.
\subsection{Closed sets}\label{closedsets}
\begin{definition}
Given a triangular choice multigraph $\mathcal{G}$ on $[m]$, a subset $S$ of $[m]$ is closed if for all distinct $x,y\in S$ and $z\not\in S$, the edge $xy$ is chosen in triangle $xyz$.  A closed set $S$ is maximal if the only proper closed subset of $[m]$ containing $S$ is $S$ itself.
\end{definition}

For example, if $\mathcal{G}$ is a TCM on $[m]$, then the empty set, all singletons and $[m]$ itself are all trivially closed. A two element set $\{x,y\}$ is closed if and only if $m_{xy}=m-2$. A three element set $\{x,y,z\}$ is closed if and only if two of $m_{xy},m_{xz},m_{yz}$ are equal to $m-3$ and the other is equal to $m-2$. Figure \ref{fig:ExampleClosedSets} illustrates the closed sets in a TCM with 5 vertices.

The following lemma shows that maximal closed sets provide a nice decomposition of $\mathcal{G}$.

\begin{lemma}\label{disjointclosedset}
Let $S_1, S_2$ be closed sets in $[m]$. If $S_1\cap S_2\neq\emptyset$, then either $S_1\subset S_2$ or $S_2\subset S_1$. Consequently, the maximal closed sets in $[m]$ form a partition of $[m]$.
\end{lemma}
\begin{proof}
Let $x \in S_1 \cap S_2$. If neither $S_1\subset S_2$ nor $S_2\subset S_1$ holds, then there exists $y\in S_1\setminus S_2$ and $z\in S_2\setminus S_1$. Now, $x,y\in S_1$ and $z\not\in S_1$ implies that in triangle $xyz$ edge $xy$ is chosen, while $x,z\in S_2$ and $y\not\in S_2$ implies that in triangle $xyz$ edge $xz$ is chosen, a contradiction.
\end{proof}

\begin{figure}[h!]
\input{ExampleClosedSets}
\end{figure}

\subsection{Every maximal closed set has size at least 2}\label{atleast2}
Recall that for every TCM $\mathcal{G}$ on $[m]$, we defined $$w(\mathcal{G})=w(\mathcal{G},2)=\sum_{xy\in[m]^{(2)}}2^{m^{\mathcal{G}}_{xy}},$$ so that $H(m)=H(m,2)=\max\{w(\mathcal{G})\colon\mathcal{G}\text{ is a TCM on }[m]\}$. The goal in this subsection is to prove that there exists a TCM $\mathcal{G}$ on $[m]$ satisfying $w(\mathcal{G})=H(m)$, such that each of its vertices is contained in a closed set of size 2 or 3. In particular, this means that its vertices are partitioned into maximal closed sets of sizes at least 2. The plan is to show that if $\mathcal{G}$ is an extremal TCM on $[m]$, then either all vertices in $\mathcal{G}$ are contained in closed sets of size 2 or 3, or we can change some edge choices to obtain another extremal TCM $\mathcal{G}'$ on $[m]$ that has strictly more vertices contained in closed sets of size 2 or 3.

The next few lemmas identify several types of changes in edge choices we could perform on a TCM $\mathcal{G}$ that does not decrease, or even strictly increase, $w(\mathcal{G})$.

\begin{lemma}\label{lemma:unique-choice}
Let $\mathcal{G}$ be a TCM on $[m]$. Suppose $x,y,z$ are distinct vertices such that $m_{xy}\geq m_{xz}-1$ and edge $xz$ is chosen in triangle $xyz$. Then swapping the edge choice from $xz$ to $xy$ does not decrease $w(\mathcal{G})$, and strictly increases it if $m_{xy}\geq m_{xz}$.

Therefore, if $\mathcal{G}$ is extremal and $x,y,z$ are distinct, then $m_{xy}>m_{xz},m_{yz}$ if and only if $xy$ is the edge chosen in triangle $xyz$.
\end{lemma}
\begin{proof}
For the first part, let $\mathcal{G}'$ be the new TCM obtained after changing the edge choice on $xyz$ from $xz$ to $xy$. Then we have 
\begin{align*}
w(\mathcal{G}')-w(\mathcal{G})&=2^{m_{xz}-1}+2^{m_{xy}+1}-2^{m_{xz}}-2^{m_{xy}}\\
&=2^{m_{xy}}-2^{m_{xz}-1}\\
&\geq2^{m_{xz}-1}-2^{m_{xz}-1}=0,    
\end{align*}
with equality only if $m_{xy}=m_{xz}-1$.

Now assume $\mathcal{G}$ is extremal. If $m_{xy}>m_{xz},m_{yz}$ and the edge $e$ chosen on $xyz$ is $xz$ or $yz$, then by the above swapping the choice from $e$ to $xy$ strictly increases $w(\mathcal{G})$, a contradiction. Conversely, if $xy$ is chosen on $xyz$ but $m_{xy}\leq m_{e}$, where $e$ is $xz$ or $yz$, then again by the above, switching the choice from $xy$ to $e$ strictly increases $w(\mathcal{G})$, a contradiction.  
\end{proof}

\begin{lemma}\label{prop:swap}
Let $\mathcal{G}$ be a TCM on $[m]$, and let $x,y,z,w$ be distinct vertices in $\mathcal{G}$. Assume that $m_{xy}=m_{xz}+1=m_{xw}+1$ and edge $xy$ is chosen in both triangles $xyz$ and $xyw$. Let $\mathcal{G}'$ be the TCM obtained swapping the choice from edge $xy$ to edge $xz$ in triangle $xyz$ and swapping from from $xy$ to $xw$ in triangle $xyw$, then $w(\mathcal{G}')>w(\mathcal{G})$.
\end{lemma}
\begin{proof}
We have 
\begin{align*}
w(\mathcal{G}')-w(\mathcal{G})&=2^{m_{xy}-2}+2^{m_{xz}+1}+2^{m_{xw}+1}-2^{m_{xy}}-2^{m_{xz}}-2^{m_{xw}}\\
&=2^{m_{xy}-2}+2^{m_{xy}}+2^{m_{xy}}-2^{m_{xy}}-2^{m_{xy}-1}-2^{m_{xy}-1}\\
&=2^{m_{xy}-2}>0,
\end{align*}
as required.
\end{proof}

\begin{lemma}\label{prop:consecutive}
Let $\mathcal{G}$ be a TCM on $[m]$, and let $t\geq4$ be an integer. Suppose there exist distinct vertices $x$ and $y_3,\cdots,y_{t-1}$ in $\mathcal{G}$ such that 
\begin{itemize}
    \item $m_{xy_i}=m-i$ for all $3\leq i\leq t-1$, 
    \item $y_3,\cdots,y_{t-1}$ all belong to distinct closed sets $Y_3,\cdots,Y_{t-1}$ of size at least $2$,
    \item $x$ is not in a closed set of size 2, and $x\not\in Y_3\cup\cdots\cup Y_{t-1}$.
\end{itemize}
Then, $|Y_i|=2$ for every $3\leq i\leq t-1$. Furthermore, if either
\begin{enumerate}
    \item there exists $y_t\in Y_i$ for some $3\leq i\leq t-1$ such that $m_{xy_t}=m-t$, or
    \item there exist distinct $y_t,y_t'$ such that $m_{xy_t}=m_{xy'_t}=m-t$.
\end{enumerate}
Then, there is a series of swaps of edge choices that strictly increases $w(\mathcal{G})$.
\end{lemma}
\begin{proof}
First, we use induction to show that for each $3\leq i\leq t-1$, $|Y_i|=2$, and the edge $xy_i$ is chosen in the triangle $xy_iy$ if and only if $y\not\in\{x,y_3,\cdots,y_{i-1}\}\cup Y_i$. For $i=3$, since $x\not\in Y_3$ and $Y_3$ is closed, the edge $xy_3$ is not chosen in triangle $xy_3y$ for any other vertex $y\in Y_3$. Since $m_{xy_3}=m-3$, we must have $|Y_3|=2$ and that $xy_3$ is chosen in triangle $xy_3y$ for all $y\not\in\{x\}\cup Y_3$. Suppose now $i\geq 4$ and this is true for all smaller $i \geq 3$. Then by the induction hypothesis, the edge $xy_i$ is not chosen in triangles $xy_iy_j$ for all $3\le j<i$ and also not in triangles $xy_iy$ for all $y\in Y_i$ not equal to $y_i$. So $m-i=m_{xy_i}\leq m-2-(i-3)-(|Y_i|-1)$, implying that $|Y_i|=2$ and that edge $xy_i$ is chosen in triangle $xy_iy$ for all $y\not\in\{x,y_3,\cdots,y_{i-1}\}\cup Y_i$, as required.

In the first case, the assumption implies, in particular, that $Y_i=\{y_i,y_t\}$. From above, $xy_t$ is not chosen in triangle $xy_ty_j$ for all $3\leq j\leq t-1$. But $m_{xy_t}=m-t$, so there exists a vertex $z\not\in\{x,y_3,\cdots, y_t\}$ such that the edge $xy_t$ is not chosen in triangle
$xy_tz$. Let $e$ be the edge that is chosen in $xy_tz$, so $e=xz$ or $y_tz$. If $m_e\leq m_{xy_t}$, then swapping the choice from $e$ to $xy_t$ strictly increases $w(\mathcal{G})$ by Lemma \ref{lemma:unique-choice}, as required, so we can assume $m_e>m_{xy_t}$. Say $m_e=m-k$ with $t>k$. Since $x$ is not in a closed set of size 2, we have $m_{xz}\leq m-3$. Since $Y_i$ is a closed set of size 2, $m_{y_iy_t}=m-2$, which means edge $y_iy_t$ is chosen in triangle $y_iy_tz$. Hence, $m_{y_tz}\leq m-3$ and so $k\geq 3$. Apply the first part of Lemma \ref{lemma:unique-choice} in triangles $xy_ty_{t-1},xy_ty_{t-2},\cdots, xy_ty_k$ successively in this order, swapping the choice from $xy_j$ to $xy_t$ for each $k\leq j\leq t-1$, which changes $m_{xy_t}$ to $m-k$ while maintaining the value of $w(\mathcal{G})$. Now we can swap the choice in triangle $xy_tz$ from $e$ to $xy_t$, which by Lemma \ref{lemma:unique-choice} strictly increases $w(\mathcal{G})$, as required.

In the second case, note that edge $xy_{t-1}$ is chosen in both triangles $xy_{t-1}y_t$ and $xy_{t-1}y_t'$ and $m_{xy_{t-1}}=m-t+1=m_{xy_t}+1=m_{xy_t'}+1$. Hence, we may apply the two swaps given by Lemma \ref{prop:swap} to strictly increase $w(\mathcal{G})$.
\end{proof}

The next lemma, roughly speaking, states that a vertex cannot have many edges with large multiplicities incident with it.
\begin{lemma}\label{lemma:weight-degree}
Let $\mathcal{G}$ be a TCM on $[m]$ and let $x$ be a vertex in $\mathcal{G}$. For every $j\geq 0$, let $d_j(x)$ denote the number of edges of multiplicity $j$ incident with $x$. Then for every $0\leq t\leq m-1$ we have 
\begin{equation*}
\sum_{j\geq t}d_j(x)\le m-1-t
\end{equation*}
\end{lemma}
\begin{proof}
If there exists no edge incident with $x$ with multiplicity at least $t$ in $\mathcal{G}$, then we are done. Otherwise, let $xy$ be an edge minimising $m_{xy}$, subjected to the condition that $m_{xy}\geq t$. Then there exists $m_{xy}\geq t$ vertices $z\not=x,y$ in $\mathcal{G}$ such that in triangle $xyz$ the edge $xy$ is chosen. On the other hand, by Lemma~\ref{lemma:unique-choice}, for any $z\not=x,y$ satisfying $m_{xz}\geq m_{xy}$, the edge $xy$ is not chosen in triangle $xyz$. There are $\sum_{j\geq m_{xy}}d_j(x)-1=\sum_{j\geq t}d_j(x)-1$ such $z$. Hence, $\sum_{j\geq t}d_j(x)-1+m_{xy}\leq m-2$, which gives $\sum_{j\geq t}d_j(x)\leq m-1-m_{xy}\leq m-1-t$.
\end{proof}

We are now ready to prove the main result of this subsection. 
\begin{prop}\label{prop:sizeatleast2}
There exists a triangular choice multigraph $\mathcal{G}$ on $[m]$ with $w(\mathcal{G})=H(m)$, such that every vertex in $\mathcal{G}$ is contained in a closed set of size 2 or 3.
\end{prop}
\begin{proof}
Let $\mathcal{G}$ be a TCM on $[m]$ satisfying $w(\mathcal{G})=H(m)$ and minimising the size of the set $S\subset[m]$ of vertices that are not contained in a closed set of size 2 or 3. It suffices to show that $|S|=0$. For a contradiction, suppose $|S|\geq 1$.

\noindent\textbf{Case 1.} $|S|\geq2$. Let $xx'$ be an edge between vertices in $S$ maximising $m_{xx'}$, and suppose $m_{xx'}=m-t$. Using Lemma \ref{lemma:unique-choice} and the maximality of $m_{xx'}$ and $w(\mathcal{G})$, we have that if $z$ is another vertex such that edge $xx'$ is not chosen on triangle $xx'z$, then $z\not\in S$. We may also assume that every edge adjacent to either $x$ or $x'$ has multiplicity at most $m-3$, as if, say, $m_{xz}=m-2$ for some $z\not=x$, then $\{x,z\}$ is a closed set of size 2, contradicting $x\in S$. In particular, $t\geq3$. If $t=3$, let $z$ be the unique vertex such that edge $xx'$ is not chosen in triangle $xx'z$. Say $xz$ is chosen in triangle $xx'z$, then by Lemma \ref{lemma:unique-choice}, $m_{xz}>m_{xx'}$, so $m_{xz}=m-2$. Thus, $\{x,z\}$ is a closed set of size 2, contradicting $x\in S$. Hence, $t\geq4$. We have two further subcases.

\textbf{Case 1.1.} For at least one of $x$ and $x'$, say $x$, there exist two vertices $z,z'$ such that $m_{xz}=m_{xz'}=m-3$. By Lemma \ref{lemma:unique-choice}, $m_{xz}=m_{xz'}$ implies that $zz'$ is the edge chosen in triangle $xzz'$ and $m_{zz'}=m-2$. It follows that $\{x,z,z'\}$ form a closed set of size 3, contradicting $x\in S$.

\textbf{Case 1.2.} For each of $x$ and $x'$, there is at most one edge of multiplicity $m-3$ incident with it. Suppose there are $r-2$ vertices $y$ such that edge $xy$ is chosen in triangle $xx'y$. Let these $r-2$ vertices be $y_3,\cdots,y_r$ and assume that $m-3\geq m_{xy_3}\geq\cdots\geq m_{xy_r}$. We claim that $$\sum_{j=3}^r2^{m_{xy_j}}\leq2^{m-2}-2^{m-r}.$$
If $r=3$ this follows from $m_{xy_3}\leq m-3$. If $r=4$, this follows from $m_{xy_4}\leq m-4$, as otherwise there are two edges of multiplicities $m-3$ incident with $x$. For the $r\geq 5$ case, first suppose that $m_{xy_j}\geq m-j$ for all $3\leq j\leq r$. Then using $w(\mathcal{G})=H(m)$ and Lemma \ref{prop:consecutive}, we see inductively that $m_{xy_j}=m-j$ for all $3\leq j\leq r$ and each $y_j$ belongs to a different closed set of size 2. Thus, $\sum_{j=3}^r2^{m_{xy_j}}=\sum_{j=3}^r2^{m-j}=2^{m-2}-2^{m-r}$. If this is not true, let $3\leq i\leq r$ be the smallest index such that $m_{xy_i}<m-i$. Again by Lemma \ref{prop:consecutive}, we have $m_{xy_j}=m-j$ for all $3\leq j\leq i-1$. By Lemma \ref{lemma:weight-degree}, for all $3\leq j\leq r$, there are at most $j-1$ edges incident with $x$ that has multiplicity at least $m-j$. Hence, the sequence $(d_j)_{j=3}^r$ given by $d_j=m-j$ for all $3\leq j\leq i-1$, $d_j=m-i-1$ for all $i\leq j\leq i+2$ and $d_j=m-j+1$ for all $i+3\leq j\leq r$ dominates the sequence $(m_{xy_j})_{j=3}^r$, in the sense that $d_j\geq m_{xy_j}$ for all $3\leq j\leq r$. Thus, $\sum_{j=3}^r2^{m_{xy_j}}\leq\sum_{j=3}^r2^{d_j}=2^{m-2}-2^{m-r+1}$, as required.

Similarly, suppose there are $s-2$ vertices $y'$ such that edge $x'y'$ is chosen in triangle $xx'y'$. Let them be $y'_3,\cdots,y'_s$ and assume that $m-3\geq m_{x'y'_3}\geq\cdots\geq m_{x'y'_s}$, then $\sum_{j=3}^s2^{m_{x'y'_j}}\leq2^{m-2}-2^{m-s}$. Note that as $y_3,\cdots,y_r,y'_3,\cdots,y'_s$ are all the vertices $y$ such that edge $xx'$ is not chosen in triangle $xx'y$, we have $r-2+s-2=t-2$. Consider the TCM $\mathcal{G}'$ obtained from $\mathcal{G}$ by choosing edge $xx'$ instead of $xy_j$ in triangles $xx'y_j$ for all $3\leq j\leq r$ and choosing edge $xx'$ instead of $x'y'_j$ in triangles $xx'y'_j$ for all $3\leq j\leq s$. Then 
\begin{align*}
w(\mathcal{G}')-w(\mathcal{G})&=2^{m-2}+\sum_{j=3}^r2^{m_{xy_j}-1}+\sum_{j=3}^s2^{m_{x'y'_j}-1}-2^{m-t}-\sum_{j=3}^r2^{m_{xy_j}}-\sum_{j=3}^s2^{m_{x'y'_j}}\\
&=2^{m-2}-2^{m-t}-\sum_{j=3}^r2^{m_{xy_j}-1}-\sum_{j=3}^s2^{m_{x'y'_j}-1}\\
&\geq2^{m-2}-2^{m-t}-(2^{m-3}-2^{m-r-1})-(2^{m-3}-2^{m-s-1})\\
&=2^{m-r-1}+2^{m-s-1}-2^{m-r-s+2}>0,
\end{align*}
contradicting $w(\mathcal{G})=H(m)$.
 
\noindent\textbf{Case 2.} $|S|=1$. Let $x$ be the unique vertex in $S$ and let $z$ be a vertex maximising $m_{xz}$. Since $|S|=1$, $z$ is in a closed set of size 2 or 3. Let $m_{xz}=m-t$. If $t=2$, we have a contradiction as $\{x,z\}$ is a size 2 closed set. So $t\geq 3$. 

\textbf{Case 2.1.} $z$ is in a closed set $C$ of size 3. Let $y_3,y_4$ be the other two vertices in $C$. Two of edges $zy_3,zy_4,y_3y_4$ have multiplicity $m-3$ while the other has multiplicity $m-2$. If $m_{y_3y_4}=m-2$, do nothing. If $m_{y_3y_4}=m-3$, then $y_3y_4$ is not chosen in triangle $zy_3y_4$. Change the choice on $zy_3y_4$ so that $y_3y_4$ is chosen. Then $m_{y_3y_4}=m-2$ and $w(\mathcal{G})$ is unchanged. In both cases, we now have $m_{zy_3}=m_{zy_4}=m-3$. Also note that $t\geq 4$ as edge $xz$ is not chosen in the triangles $xzy_3$, $xzy_4$. There are $t-2$ vertices $y$ such that edge $xz$ is not chosen in triangle $xzy$. Using $w(\mathcal{G})=H(m)$ and Lemma \ref{lemma:unique-choice}, we see that for every such $y$, edge $zy$ is chosen in triangle $xzy$, as otherwise $m_{xy}>m_{xz}$, contradicting the maximality of $m_{xz}$. Let these $t-2$ vertices be $y_3,\cdots,y_t$ and assume that $m_{zy_3}\geq\cdots\geq m_{zy_t}$. By Lemma~\ref{lemma:weight-degree}, $(m_{zy_j})_{j=3}^t$ is dominated by the sequence $(d_j)_{j=3}^t$, where $d_3=m-3$ and $d_j=m-j+1$ for all $4\leq j\leq t$. Thus, $\sum_{j=3}^t2^{m_{zy_j}}\leq\sum_{j=3}^t2^{d_j}=2^{m-2}+2^{m-3}-2^{m-t+1}$. Consider the TCM $\mathcal{G}'$ obtained from $\mathcal{G}$ by picking edge $xz$ instead of $zy_j$ in triangle $xzy_j$ for every $3\leq j\leq t$. Then
\begin{align*}
w(\mathcal{G}')-w(\mathcal{G})&=2^{m-2}+\sum_{j=3}^t2^{m_{zy_j}-1}-2^{m-t}-\sum_{j=3}^t2^{m_{zy_j}}\\
&=2^{m-2}-2^{m-t}-\sum_{j=3}^t2^{m_{zy_j}-1}\\
&\geq2^{m-2}-2^{m-t}-2^{m-3}-2^{m-4}+2^{m-t}\\
&=2^{m-4}>0,     
\end{align*}
contradicting $w(\mathcal{G})=H(m)$.   

\textbf{Case 2.2.} $z$ is not in a closed set of size 3, so it is in a closed set $C$ of size 2. Let $z'$ be the other vertex in $C$. Like in Case 2.1, if $y$ is a vertex such that edge $xz$ is not chosen in triangle $xzy$, then $zy$ must be the edge chosen. Let the $t-2$ vertices $y$ such that edge $zy$ is chosen in triangle $xzy$ be $y_3,\cdots,y_t$, and assume that $m_{zy_3}\geq\cdots\geq m_{zy_t}$. Then $z'=y_3$ as $m_{zz'}=m-2$. We claim that $t=3$. Indeed, if $t\ge 4$, let $C'$ be a closed set of size 2 or 3 containing $y_4$. $C'$ must be disjoint from $C$ as otherwise Lemma \ref{disjointclosedset} implies that $C'=C\cup\{y_4\}=\{z,z',y_4\}$, contradicting that $z$ is not contained in a closed set of size 3. Thus, $m_{zy_4}\leq m-4$ as edge $zy_4$ is not chosen in triangle $zy_4z'$ and $zy_4y_4'$, where $y_4'$ is any vertex in $C'$ other than $y_4$. By Lemma~\ref{lemma:weight-degree}, $(m_{zy_j})_{j=3}^t$ is dominated by the sequence $(d_j)_{j=3}^t$, where $d_3=m-2$, $d_4=m-4$ and $d_j=m-j+1$ for all $5\leq j\leq t$. Thus, $\sum_{j=4}^t2^{m_{zy_j}}\leq\sum_{j=4}^t2^{d_j}=2^{m-3}+2^{m-4}-2^{m-t+1}$. Consider the TCM $\mathcal{G}'$ obtained from $\mathcal{G}$ by picking edge $xz$ instead of $zy_j$ in triangles $xzy_j$ for every $4\leq j\leq t$. Then 
\begin{align*}
w(\mathcal{G}')-w(\mathcal{G})&=2^{m-3}+\sum_{j=4}^t2^{m_{zy_j}-1}-2^{m-t}-\sum_{j=4}^t2^{m_{zy_j}}\\
&=2^{m-3}-2^{m-t}-\sum_{j=4}^t2^{m_{zy_j}-1}\\
&\geq2^{m-3}-2^{m-t}-2^{m-4}-2^{m-5}+2^{m-t}\\
&=2^{m-5}>0,
\end{align*}
contradicting $w(\mathcal{G})=H(m)$.

Hence, $t=3$ and $m_{xz}=m-3$. If $m_{xz'}=m-3$, then $\{x,z,z'\}$ is a closed set of size 3, a contradiction. Thus, $m_{xz'}=m-t'$ for some $t'\geq 4$. Let $w_3,\cdots,w_r$ be all the vertices such that edge $xw_j$ is chosen in the triangle $xz'w_j$ for all $3\leq j\leq r$, and assume that $m_{xw_3}\geq\cdots\geq m_{xw_r}$. In particular, $w_3=z$ and $m-3=m_{xw_3}$. We claim that $$\sum_{j=4}^r2^{m_{xw_j}}\leq2^{m-3}-2^{m-r}.$$ This is trivial if $r=3$. If $r\geq 4$, since $m_{xz}=m-3$ and $m_{zz'}=m-2$, the edge $xz$ is chosen in triangle $xzw$ for all $w\not=z'$. So by Lemma \ref{lemma:unique-choice}, $m_{xw_4}\leq m_{xz}-1=m-4$, which in particular proves the case when $r=4$. Suppose now $r\geq 5$. If $m_{xw_4}=m_{xw_5}=m-4$, then by Lemma \ref{prop:swap}, there is a swap of choice that strictly increase $w(\mathcal{G})$, a contradiction. Thus, $m_{xw_5}\leq m-5$, and this allows us to use the same argument as in Case 1.2 to show that $\sum_{j=4}^r2^{m_{xw_j}}\leq2^{m-3}-2^{m-r}$. 

Now let $w'_3,\cdots,w'_s$ be all the vertices such that edge $z'w'_j$ is chosen in the triangle $xz'w'_j$ for all $3\leq j\leq s$, and assume that $m_{z'w'_3}\geq\cdots\geq m_{z'w'_s}$. In particular, $w'_3=z$ and $m_{z'w'_3}=m-2$. We claim that $$\sum_{j=4}^s2^{m_{z'w'_j}}\leq2^{m-3}-2^{m-s}.$$ 
This is trivial if $s=3$. If $s\geq 4$, $m_{z'w'_4}\not=m-2$ as $m_{z'z}=m-2$. If $m_{z'w'_4}=m-3$, then as $m_{z'z}=m-2$ and $m_{xz}=m-3$, Lemma \ref{prop:swap} implies there is an edge swap that strictly increases $w(\mathcal{G})$, a contradiction. Hence $m_{z'w'_4}\leq m-4$, and the $s=4$ case is proved. Assume now $s\geq 5$. If $m_{z'w'_4}=m_{z'w'_5}=m-4$, then we change the edge choice on triangle $xzz'$ from $zz'$ to $xz$. This maintains the value of $w(\mathcal{G})$, but now $m_{zz'}=m-3$, so we can apply Lemma \ref{prop:swap} to find swap of choice that strictly increase $w(\mathcal{G})$, a contradiction. Thus, $m_{z'w'_5}\leq m-5$ and again this allows us to apply the same argument as in Case 1.2 to show that $\sum_{j=4}^s2^{m_{z'w'_j}}\leq2^{m-3}-2^{m-s}$. Note that $r-3+s-3=t'-3$. Consider the TCM $\mathcal{G}'$ obtained from $\mathcal{G}$ by choosing edge $xz'$ instead of $xw_j$ in triangles $xz'w_j$ for all $4\leq j\leq r$ and choosing edge $xz'$ instead of $z'w'_j$ in triangles $xz'w'_j$ for all $4\leq j\leq s$. We have
 \begin{align*}
 w(\mathcal{G}')-w(\mathcal{G})&=2^{m-3}+\sum_{j=4}^r2^{m_{xw_j}-1}+\sum_{j=4}^s2^{m_{z'w'_j}-1}-2^{m-t'}-\sum_{j=4}^r2^{m_{xw_j}}-\sum_{j=4}^s2^{m_{z'w'_j}}\\
 &=2^{m-3}-2^{m-t'}-\sum_{j=4}^r2^{m_{xw_j}-1}-\sum_{j=4}^s2^{m_{z'w'_j}-1}\\
 &\geq2^{m-3}-2^{m-t'}-(2^{m-4}-2^{m-r-1})-(2^{m-4}-2^{m-s-1})\\
 &=2^{m-r-1}+2^{m-s-1}-2^{m-r-s+3}>0,
 \end{align*}
 contradicting $w(\mathcal{G})=H(m)$. 
\end{proof}

\subsection{Proofs of Theorem \ref{thm:upperbds} and Theorem \ref{thm:generalupperbds}}\label{upperproof}
In this subsection, we prove Theorem \ref{thm:upperbds} and Theorem \ref{thm:generalupperbds}. We begin by proving several inequalities that will be needed in the proof of Theorem \ref{thm:upperbds}, all of which are consequences of the well-known Karamata's Inequality below. 

\begin{lemma}[Karamata's Inequality, \cite{JovanKaramata1932}]\label{karamata} 
Let $f$ be a real-valued convex function defined on some interval $I$. Suppose $\lambda_1\geq\cdots\geq \lambda_n$ and $\mu_1\geq\cdots\geq \mu_n$ are numbers in $I$, such that for every $k\in[n]$, we have
$$\sum_{i=1}^{k}\lambda_i\le\sum_{i=1}^{k}\mu_i,$$
then 
$$\sum_{i=1}^nf(\lambda_i)\le\sum_{i=1}^nf(\mu_i).$$   
\end{lemma}

\begin{lemma}\label{lem:majorise}
Let $a_1\geq\cdots\geq a_n\geq 0$, $\lambda_1\geq\cdots\geq \lambda_n$ and $\mu_1\geq\cdots\geq \mu_n$. If for every $k\in[n]$, we have
$$\sum_{i=1}^{k}\lambda_i\le\sum_{i=1}^{k}\mu_i,$$
then 
$$\sum_{i=1}^n\lambda_ia_i\le\sum_{i=1}^n\mu_ia_i.$$   
\end{lemma}
\begin{proof}
Set $a_{n+1}=0$. For each $k\in[n]$, let $f_k(x)=(a_k-a_{k+1})x$, and note that $f_i$ is convex. Then by Lemma \ref{karamata}, we have
$$\sum_{i=1}^k(a_k-a_{k+1})\lambda_i=\sum_{i=1}^kf_k(\lambda_i)\leq\sum_{i=1}^kf_k(\mu_i)=\sum_{i=1}^k(a_k-a_{k+1})\mu_i,$$
for each $k\in[n]$. Summing these together, we obtain
$$\sum_{i=1}^na_i\lambda_i=\sum_{k=1}^n\sum_{i=1}^k(a_k-a_{k+1})\lambda_i\leq\sum_{k=1}^n\sum_{i=1}^k(a_k-a_{k+1})\mu_i=\sum_{i=1}^na_i\mu_i,$$
as required.
\end{proof}

\begin{lemma}\label{minimisexiai}
Let $n>n'\geq 1$, $a_k\geq\cdots\geq a_1>0$ and $x_1,\cdots,x_k\geq0$, such that $\sum_{j=1}^ka_j=n$, $\sum_{j=1}^kx_j=n'$ and $a_j\geq x_j$ for all $j\in[k]$. Let $q\in[k]$ be the unique index satisfying $\sum_{j=1}^{q-1}a_j\leq n'<\sum_{j=1}^qa_j$, and let $y_j=a_j$ for all $j\in[q-1]$, $y_q=n'-\sum_{j=1}^{q-1}a_j$, and $y_j=0$ for all $j\in[k]\setminus[q]$. Then
$$\sum_{j=1}^k(a_j-x_j)^2\leq\sum_{j=1}^k(a_j-y_j)^2.$$
\end{lemma}
\begin{proof}
Let $\sigma$ be a permutation of $[k]$ such that $a_{\sigma(1)}-x_{\sigma(1)}\leq\cdots\leq a_{\sigma(k)}-x_{\sigma(k)}$. For each $k'\in[q]$, we have $$\sum_{j=k'}^{k}(a_{\sigma(j)}-x_{\sigma(j)})\leq\sum_{j=1}^{k}(a_{\sigma(j)}-x_{\sigma(j)})=\sum_{j=1}^{k}(a_j-x_j)=n-n'=\sum_{j=k'}^k(a_j-y_j),$$
as $a_j=y_j$ for all $j\in[q-1]$. While for each $k'\in[k]\setminus[q]$, $$\sum_{j=k'}^k(a_{\sigma(j)}-x_{\sigma(j)})\leq\sum_{j=k'}^ka_{\sigma(j)}\leq\sum_{j=k'}^ka_j =\sum_{j=k'}^k(a_j-y_j).$$ Hence, applying Lemma \ref{karamata} with the index order reversed, and using $f(x)=x^2$ is convex, we have $$\sum_{j=1}^k(a_j-x_j)^2=\sum_{j=1}^k(a_{\sigma(j)}-x_{\sigma(j)})^2\leq\sum_{j=1}^k(a_j-y_j)^2,$$
as required.
\end{proof}

\begin{lemma}\label{vertexcontribution}
Let $\mathcal{G}$ be a TCM. Let $A_1,\cdots,A_k$ be the maximal closed sets in $\mathcal{G}$ and suppose they have sizes $a_1\leq\cdots\leq a_k$, respectively. Fix an $i\in[k]$ and $v_i\in A_i$. For each $j\in[k]\setminus\{i\}$, let $s_j=\sum_{j'\in[j]\setminus\{i\}}a_{j'}$. Then, we have 
\begin{equation}\label{4.10statement}
\sum_{u\not\in A_i }2^{m^\mathcal{G}_{uv_i}}\leq\sum_{j\in[k]\setminus\{i\}}a_j2^{m-a_i-s_j}.
\end{equation}
\end{lemma}
\begin{proof}
Let $n=m-a_i$. Order the $n$ vertices $u\not\in A_i$ in decreasing order of $m_{uv_i}^{\mathcal{G}}$ as $u_1,\cdots,u_n$, and let $\lambda_\ell=m^{\mathcal{G}}_{u_\ell v_i}$ for all $\ell\in[n]$. For every $j\in[k]\setminus\{i\}$, let $\mu_\ell=m-a_i-s_j$ for all $1+s_j-a_j\leq\ell\leq s_j$. Then (\ref{4.10statement}) is equivalent to
$$\sum_{\ell=1}^n2^{\lambda_\ell}\leq\sum_{\ell=1}^n2^{\mu_\ell}.$$
Since both $(\lambda_\ell)$ and $(\mu_\ell)$ are decreasing sequences and $2^x$ is a convex function, by Lemma \ref{karamata}, it suffices to show for all $n'\in[n]$ we have
$$\sum_{\ell=1}^{n'}\lambda_\ell\leq\sum_{\ell=1}^{n'}\mu_\ell.$$

Fix any $n'\in[n]$. Suppose that among $u_1,\cdots,u_{n'}$, there are $x_j$ of these vertices belonging to $A_j$ for each $j\in[k]\setminus\{i\}$. Then, we claim that
\begin{equation}\label{eq:q-multiplicity}
    \sum_{\ell=1}^{n'} \lambda_\ell\le \sum_{j\in[k]\setminus\{i\}}x_j(m-a_j-a_i) - \sum_{\substack{1\le j<j'\le k\\j,j'\not=i}}x_jx_{j'}.
\end{equation}
Indeed, since $A_i$ and $A_j$ are closed sets, if $u_\ell\in A_j$, then $\lambda_\ell=m^{\mathcal{G}}_{u_\ell v_i}\leq m-a_j-a_i$. Furthermore, for any distinct $j,j'\in[k]$ not equal to $i$, if $u_{\ell}\in A_j$ and $u_{\ell'}\in A_{j'}$, then at least one of edge $u_\ell v_i$ and $u_{\ell'}v_i$ is not chosen in triangle $u_\ell u_{\ell'}v_i$, proving (\ref{eq:q-multiplicity}).

If $n=n'$, then $x_j=a_j$ for all $j\in[k]\setminus\{i\}$ and (\ref{eq:q-multiplicity}) implies
\begin{align*}
\sum_{\ell=1}^{n} \lambda_\ell&\le \sum_{j\in[k]\setminus\{i\}}a_j(m-a_j-a_i) - \sum_{\substack{1\le j<j'\le k\\j,j'\not=i}}a_ja_{j'}\\
&=\sum_{j\in[k]\setminus\{i\}}a_j(m-a_i-s_j)=\sum_{\ell=1}^n\mu_\ell,
\end{align*}
as required. If $n>n'$, Continuing from (\ref{eq:q-multiplicity}), we have
\begin{align*}
\sum_{\ell=1}^{n'} \lambda_\ell&\le \sum_{j\in[k]\setminus\{i\}}x_j(m-a_j-a_i) - \sum_{\substack{1\le j<j'\le k\\j,j'\not=i}}x_jx_{j'}\\
&=\sum_{j\in[k]\setminus\{i\}}x_j(m-a_i)-\sum_{j\in[k]\setminus\{i\}}x_ja_j-\frac12\left(\sum_{j\in[k]\setminus\{i\}}x_j\right)^2+\frac12\sum_{j\in[k]\setminus\{i\}}x_j^2\\
&=n'(m-a_i)-\frac12n'^2-\frac12\sum_{j\in[k]\setminus\{i\}}a_j^2+\frac12\sum_{j\in[k]\setminus\{i\}}(a_j-x_j)^2
\end{align*}
Hence, if we let $q\in[k]\setminus\{i\}$ be the unique index satisfying $s_q-a_q\leq n'<s_q$, then by Lemma \ref{minimisexiai}, the maximum on the right hand side of (\ref{eq:q-multiplicity}) is achieved when $x_j=a_j$ for all $j\in[q-1]\setminus\{i\}$, $x_q=n'-s_q+a_q$, and $x_j=0$ for all $j\in[k]\setminus([q]\cup\{i\})$. Therefore, we have 
\begin{align*}
\sum_{\ell=1}^{n'} \lambda_\ell&\le x_q(m-a_q-a_i)+\sum_{\substack{j\in[q-1]\\j\not=i}}a_j(m-a_j-a_i)-\sum_{\substack{j\in[q-1]\\j\not=i}}a_jx_q- \sum_{\substack{1\le j<j'\le q-1\\j,j'\not=i}}a_ja_{j'}\\
&=x_q(m-a_i-s_q)+\sum_{\substack{j\in[q-1]\\j\not=i}}a_j(m-a_i-s_j)=\sum_{\ell=1}^{n'}\mu_\ell,
\end{align*}
which proves the lemma.
\end{proof}

We are now ready to prove Theorem \ref{thm:upperbds}.
\begin{proof}[Proof of Theorem \ref{thm:upperbds}]
Let $\mathcal{G}$ be a TCM on $[m]$ with $w(\mathcal{G})=H(m)$. Let the maximal closed sets in $\mathcal{G}$ be $A_1,\cdots,A_k$. For each $i\in[k]$, let $|A_i|=a_i$ so that $\sum_{i=1}^ka_i=m$, and assume $a_1\leq\cdots\leq a_k$. By Proposition \ref{prop:sizeatleast2} we may assume $a_i\geq2$ for all $i\in[k]$. Recall that $h(m)=\frac{H(m)}{m2^{m-1}}$, we will show $h(m)\leq\frac{83}{192}$.

We can divide every edge in $\mathcal{G}$ into two types, either it is within a closed set, or it goes between two closed sets. For each $i\in[k]$, all edges within $A_i$ together contribute at most $$\sum_{xy\in A_i^{(2)}}2^{m_{xy}^\mathcal{G}}=\sum_{xy\in A_i^{(2)}}2^{m_{xy}^{\mathcal{G}[A_i]}+m-a_i}\leq H(a_i)2^{m-a_i}.$$ For each $i\in[k]$ and any vertex $v_i\in A_i$, by Lemma \ref{vertexcontribution}, edges of the form $v_iu$ with $u\not\in A_i$ contribute at most $$\sum_{1\leq j<i}a_j2^{m-a_i-\sum_{\ell=1}^ja_\ell}+\sum_{i<j\leq k}a_j2^{m-\sum_{\ell=1}^ja_\ell}.$$
Hence, the total contribution from all edges going across closed sets is at most
\begin{align*}
    N&=\frac12\sum_{i=1}^ka_i\left(\sum_{1\leq j<i}a_j2^{m-a_i-\sum_{\ell\in[j]}a_{\ell}}+\sum_{i<j\leq k}a_j2^{m-\sum_{\ell\in[j]}a_{\ell}}\right)\\
    &=2^{m-1}\sum_{1\leq i<j\leq k}a_ia_j\left(\frac{1}{2^{\sum_{\ell=1}^ja_{\ell}}}+\frac{1}{2^{a_j+\sum_{\ell=1}^ia_{\ell}}}\right)\\
    &=2^{m-1}\sum_{j=2}^k\frac{a_j}{2^{a_j}}\sum_{i=1}^{j-1}\left(\frac{a_i}{2^{\sum_{\ell=1}^ia_{\ell}}}+\frac{a_i}{2^{\sum_{\ell=1}^{j-1}a_{\ell}}}\right),
\end{align*}
where we have $\frac12$ factor in front of the first sum as it counts the contribution from each edge twice. We claim that $$N\leq 2^{m-1}\sum_{j=1}^k\frac{2}{3}\frac{a_j}{2^{a_j}},$$
where observe that we have added an additional $j=1$ term in this sum. First note that $(a_j/2^{a_j})$ is a decreasing sequence. Let $\mu_2=\frac43$, and let $\mu_j=\frac23$ for all $3\leq j\leq k$. Let $$\lambda_j=\sum_{i=1}^{j-1}\left(\frac{a_i}{2^{\sum_{\ell=1}^ia_{\ell}}}+\frac{a_i}{2^{\sum_{\ell=1}^{j-1}a_{\ell}}}\right)$$
for all $2\leq j\leq k$. Note that since both $(\mu_j)$ and $(\lambda_j)$ are decreasing sequences, it suffices to show that 
\begin{equation}\label{eq:partialsum}
\sum_{j=2}^{k'}\lambda_j\leq\frac{2k'}{3}=\sum_{j=2}^{k'}\mu_j
\end{equation}
for all $2\leq k'\leq k$, as then by Lemma \ref{lem:majorise}, we have $$N=2^{m-1}\sum_{j=2}^k\lambda_j\frac{a_j}{2^{a_j}}\leq2^{m-1}\sum_{j=2}^k\mu_j\frac{a_j}{2^{a_j}}\leq2^{m-1}\sum_{j=1}^k\frac23\frac{a_j}{2^{a_j}}.$$ 
We now prove (\ref{eq:partialsum}). Indeed, for every $2\leq k'\leq k$, we have
\begin{align*}
\sum_{j=2}^{k'}\lambda_j&=\sum_{j=2}^{k'}\sum_{i=1}^{j-1}\frac{a_i}{2^{\sum_{\ell=1}^ia_{\ell}}}+\sum_{j=2}^{k'}\sum_{i=1}^{j-1}\frac{a_i}{2^{\sum_{\ell=1}^{j-1}a_{\ell}}}\\
&=\sum_{i=1}^{k'-1}\frac{(k'-i)a_i}{2^{\sum_{\ell=1}^ia_{\ell}}}+\sum_{j=2}^{k'}\frac{\sum_{i=1}^{j-1}a_{\ell}}{2^{\sum_{\ell=1}^{j-1}a_{\ell}}}.
\end{align*}
We claim that if $a_t\geq 3$ for some $t\in[k']$, then both sums increase if we decrease $a_t$ by 1. Indeed, this increases the second sum as $f(x)=x/2^x$ is a decreasing function. This also increases the first sum as 
$$\frac{(k'-t)a_t}{2^{\sum_{\ell=1}^ta_{\ell}}}\leq\frac{(k'-t)(a_t-1)}{2^{\sum_{\ell=1}^ta_{\ell}-1}}$$
and $$\frac{(k'-i)a_i}{2^{\sum_{\ell=1}^ia_{\ell}}}\leq\frac{(k'-i)a_i}{2^{\sum_{\ell=1}^ia_{\ell}-1}}$$
for all $t<i\leq k'-1$. Therefore, by repeatedly applying this until $a_i=2$ for all $i\in[k']$, we have
\begin{align*}
\sum_{j=2}^{k'}\lambda_j&=\sum_{i=1}^{k'-1}\frac{(k'-i)a_i}{2^{\sum_{\ell=1}^ia_{\ell}}}+\sum_{j=2}^{k'}\frac{\sum_{i=1}^{j-1}a_{\ell}}{2^{\sum_{\ell=1}^{j-1}a_{\ell}}}\\
&\leq\sum_{i=1}^{k'-1}\frac{2k'-2i}{2^{2i}}+\sum_{j=2}^{k'}\frac{2j-2}{2^{2j-2}}\\
&=\sum_{j=1}^{k'-1}\frac{2k'-2j}{2^{2j}}+\sum_{j=1}^{k'-1}\frac{2j}{2^{2j}}\\
&=\sum_{j=1}^{k'-1}\frac{2k'}{2^{2j}} \leq \frac{2k'}3=\sum_{j=2}^{k'}\mu_j,
\end{align*}
as claimed.


It follows that $$H(m)\leq\max\left\{\sum_{i=1}^kH(a_i)2^{m-a_i}+2^{m-1}\sum_{i=1}^k\frac{2}{3}\frac{a_i}{2^{a_i}}\colon\sum_{i=1}^ka_i=m, a_i\geq2\mbox{ }\forall i\in[k]\right\},$$
and therefore 
\begin{align*}
h(m)&\leq\max\left\{\sum_{i=1}^k\frac{a_ih(a_i)}{m}+\sum_{i=1}^k\frac{2}{3}\frac{a_i}{m2^{a_i}}\colon\sum_{i=1}^ka_i=m, a_i\geq2\mbox{ }\forall i\in[k]\right\}\\
&=\max\left\{\sum_{i=1}^k\frac{a_i}{m}\left(h(a_i)+\frac{2}{3}\frac1{2^{a_i}}\right)\colon\sum_{i=1}^ka_i=m, a_i\geq2\mbox{ }\forall i\in[k]\right\}.
\end{align*}

Define a function $\varepsilon:\mathbb{N}\to\mathbb{R}$ by $\varepsilon(m)=0$ for $m\in[6]$, and $\varepsilon(m)=\sum_{j=7}^m\frac{1}{2^j}$ for $m\geq7$. We use induction to show that $h(m)+\frac{2}{3}\frac{1}{2^m}\leq\frac{5}{12}+\varepsilon(m)$ for all $m\geq2$. The base cases when $1\leq m\leq 6$ follows from the following table, which shows $h(m)+\frac{2}{3}\frac1{2^m}$ is maximised by $m=2$ in this range. 

\begin{table}[h]
\centering
\begin{tabular}{ |c|c|c|c|c|c|c|c|c| }
 \hline
 $m$ &1 & 2 & 3 & 4 & 5 & 6 \\
 \hline
 $h(m)$ & 0 & 0.25 & 0.333 & 0.375 & 0.375 & 0.380 \\
 \hline
 $h(m)+\frac{2}{3}\frac1{2^m}$ & 0.333 & 0.417 & 0.417 & 0.417 & 0.396 & 0.391 \\
 \hline
\end{tabular}
\end{table}
Assume this is true for all smaller $m$, and let $a_1,\cdots,a_k$ be such that $\sum_{i=1}^ka_i=m$ and $a_i\geq2$ for all $i\in[k]$. Then by induction hypothesis, we have
\begin{align*}
    \sum_{i=1}^k\frac{a_i}{m}\left(h(a_i)+\frac{2}{3}\frac1{2^{a_i}}\right)+\frac{2}{3}\frac{1}{2^m}&\leq\sum_{i=1}^k\frac{a_i}{m}\left(\frac{5}{12}+\varepsilon(a_i)\right)+\frac{2}{3}\frac{1}{2^m}\\
    &=\frac{5}{12}+\sum_{i=7}^{k}\frac{a_i}m\sum_{j=7}^{a_i}\frac1{2^j}+\frac{2}{3}\frac{1}{2^m}\\
    &=\frac{5}{12}+\sum_{j=7}^{m-1}\frac1{2^j}\sum_{i:a_i\geq j}\frac{a_i}m+\frac{2}{3}\frac{1}{2^m}\\
    &\leq\frac{5}{12}+\sum_{j=7}^{m-1}\frac1{2^j}+\frac{1}{2^m}\\
    &\leq\frac{5}{12}+\varepsilon(m).
\end{align*}
Since this is true for any such $a_1,\cdots,a_k$, we get $h(m)+\frac{2}{3}\frac1{2^m}\leq\frac{5}{12}+\varepsilon(m)$, as required. Finally, note that $\varepsilon(m)<\frac1{2^6}$ for all $m$, hence $h(m)\leq\frac{5}{12}+\varepsilon(m)-\frac{2}{3}\frac1{2^m}<\frac{83}{192}$ for all $m$. 
\end{proof}

We deduce Theorem \ref{thm:generalupperbds} as a corollary of Theorem \ref{thm:upperbds}.
\begin{proof}[Proof of Theorem \ref{thm:generalupperbds}]
The case when $r=3$ follows from Theorem \ref{main} and Theorem \ref{thm:upperbds}. Now let $r\geq 4$, and let $A\in\Avoid(m,r,M)$. 

For every column $c$ in $A$, define the \emph{3-support} of $c$ to be the index set of its entries that are equal to 0, 1 or 2. For every $X\subset [m]$ of size $k$, let $C^*_3(X)$ be the set of columns in $A$ with 3-support equal to $X$, and let $C_3(X)$ be the set of distinct (0,1,2)-columns formed by restricting columns in $C^*_3(X)$ to rows with indices in $X$. Viewing $C_3(X)$ as a $k$-rowed $(0,1,2)$-matrix, we must have $C_3(X)\in\Avoid(k,3,M)$. Hence, $|C_3(X)|\leq\forb(k,3,M)$, and thus $|C^*_3(X)|\leq\forb(k,3,M)(r-3)^{m-k}$. Therefore, we have
\begin{align*}
|A|&=\sum_{X\subset[m]}|C^*_3(X)|\leq\sum_{k=0}^m\binom{m}k\forb(k,3,M)(r-3)^{m-k}\\
&\leq\sum_{k=0}^m\binom{m}k\left(1+\frac{83}{192}\right)k2^{k-1}(r-3)^{m-k}+\sum_{k=0}^m\binom{m}k2^k(r-3)^{m-k}\\
&=\left(1+\frac{83}{192}\right)m(r-1)^{m-1}+(r-1)^m.
\end{align*}
Since this is true for all $A\in\Avoid(m,r,M)$, we have $\forb(m,r,M)\leq\left(1+\frac{83}{192}\right)m(r-1)^{m-1}+(r-1)^m$, as required.
\end{proof}

\section{Lower bounds}\label{lowersection}
In this section we prove Theorem \ref{thm:generallowerbds}, which bounds $H_2(m,\alpha)$ and $\forb(m,r,M)$ from below with an explicit construction, and Theorem \ref{thm:lowerbds}, which determines $H_2(m,2)$. If $\alpha\not=0$, let $h_2(m,\alpha)=2H_2(m,\alpha)/m\alpha^m$ for all $m\geq 1$. We begin by showing
$H_{2}(m,\alpha)$ and $h_{2}(m,\alpha)$ satisfy the following recurrence relations. 
\begin{lemma}\label{hrbrecurrence}
For all $m\geq 3$, we have $$H_{2}(m,\alpha)=\max\left\{H_{2}(a,\alpha)\alpha^{m-a}+H_{2}(b,\alpha)\alpha^{m-b}+ab\colon a+b=m\right\},$$
$$h_2(m,\alpha)=\max\left\{\frac{ah_2(a,\alpha)}m+\frac{bh_2(b,\alpha)}m+\frac{2ab}{m\alpha^m}\colon a+b=m\right\}.$$
\end{lemma}
\begin{proof}
Let $\mathcal{G}$ be a 2-recursive TCM on $[m]$ with $w(\mathcal{G},\alpha)=H_{2}(m,\alpha)$. Let $m=V_1\cup V_2$, $\mathcal{G}_1$, $\mathcal{G}_2$ and  $\mathcal{G}_3$ be as in Definition \ref{tcmdef}, and let $m'=|V_1|$. Then $m^{\mathcal{G}}_{xy}=m^{\mathcal{G}_1}_{xy}+m-m'$ for all distinct $x,y\in V_1$, $m^{\mathcal{G}}_{xy}=m^{\mathcal{G}_2}_{xy}+m'$ for all distinct $x,y\in V_2$, and $m^{\mathcal{G}}_{xy}=0$ for all $x\in V_1,y\in V_2$. Hence 
\begin{align*}
H_{2}(m,\alpha)&=w(\mathcal{G},\alpha)=\sum_{xy\in[m]^{(2)}}\alpha^{m^{\mathcal{G}}_{xy}}\\
&=\sum_{xy\in V_1^{(2)}}\alpha^{m^{\mathcal{G}_1}_{xy}+m-m'}+\sum_{xy\in V_2^{(2)}}\alpha^{m^{\mathcal{G}_2}_{xy}+m'}+\sum_{x\in V_1,y\in V_2}1\\
&\leq H_{2}(m',\alpha)\alpha^{m-m'}+H_{2}(m-m',\alpha)\alpha^{m'}+m'(m-m'),
\end{align*}
which is at most the right hand side. 

On the other hand, for every $a+b=m$, let $\mathcal{G}_a,\mathcal{G}_b$ be 2-recursive triangular multigraphs on vertex sets $[a]$ and $[m]\setminus[a]$, respectively, with $w(\mathcal{G}_a)=H_{2}(a,\alpha)$ and $w(\mathcal{G}_b)=H_{2}(b,\alpha)$. Let $\mathcal{G}$ be the 2-recursive 
TCM formed as in Definition \ref{tcmdef} by taking $[m]=[a]\cup([m]\setminus[a])$, $\mathcal{G}_1=\mathcal{G}_a$ and $\mathcal{G}_2=\mathcal{G}_b$. Then a calculation similar to above shows
$$H_{2}(a,\alpha)\alpha^{m-a}+H_{2}(b,\alpha)\alpha^{m-b}+ab=\sum_{xy\in[m]^{(2)}}\alpha^{m^\mathcal{G}_{xy}}=w(\mathcal{G},\alpha)\leq H_{2}(m,\alpha),$$
proving the recurrence formula for $H_2(m,\alpha)$. The recurrence formula for $h_2(m,\alpha)$ follows after dividing through by $m\alpha^m/2$.
\end{proof}

Lemma \ref{hrbrecurrence} implies that to compute $H_2(m,\alpha)$ and $h_2(m,\alpha)$, it suffices to determine for each $m\geq 3$, which split of $m$ into $a+b$ achieves the maximum on the right hand side of the recurrence formulas. In Section \ref{lowerboundconstruction}, we prove Theorem \ref{thm:generallowerbds} by obtaining lower bounds for these two quantities using a specific split. Then, in Section \ref{h2exact}, we prove Theorem \ref{thm:lowerbds} by showing that this split is indeed optimal in the case when $r=3$ and $\alpha=2$. 

\subsection{Lower bound construction}\label{lowerboundconstruction}
For every $m\geq 3$, let $k=k(m)$ be the unique integer such that $2^k+2^{k-1}\leq m<2^k+2^{k+1}$. Define inductively a 2-recursive TCM $\mathcal{G}(m)$ on $[m]$ for every $m\geq 1$ as follows. For $m=1,2$, let $\mathcal{G}(m)$ be the empty graph. For $m\geq 3$, let $\mathcal{G}(m)$ be the 2-recursive TCM formed as in Definition \ref{tcmdef} by taking $[m]=[2^k]\cup([m]\setminus [2^k])$, $\mathcal{G}_1=\mathcal{G}(2^k)$ and taking $\mathcal{G}_2$ isomorphic to $\mathcal{G}(m-2^k)$. 

For every $m\geq 1$ and $\alpha\not=0$, let $\overline{H}_2(m,\alpha)=\sum_{xy\in[m]^{2}}\alpha^{m^{\mathcal{G}(m)}_{xy}}$ and $\overline{h}_2(m,\alpha)=2\overline{H}_2(m,\alpha)/m\alpha^m$. By the construction above, and using calculation similar to the proof of Lemma \ref{hrbrecurrence} above, we have
\begin{equation}\label{h2barrecurrence}
\overline{h}_2(m,\alpha)=\frac{2^k\overline{h}_2(2^k,\alpha)}{m}+\frac{(m-2^k)\overline{h}_2(m-2^k,\alpha)}{m}+\frac{2^{k+1}(m-2^k)}{m\alpha^m}. 
\end{equation}
Moreover, since each $\mathcal{G}(m)$ is a 2-recursive TCM, we have that $\overline{H}_2(m,\alpha),\overline{h}_2(m,\alpha)$ are lower bounds for $H_2(m,\alpha)$ and $h_2(m,\alpha)$, respectively.

We first show that $\overline{h}_2(m,\alpha)$ has a nice formula when $m$ is a power of 2.
\begin{lemma}\label{2knice}
For every $k\geq 0$ and $\alpha\not=0$, $\overline{h}_2(2^k,\alpha)=\sum_{j=1}^k2^{j-1}/\alpha^{2^j}$.
\end{lemma}
\begin{proof}
We use induction on $k$. The cases when $k=0,1$ follow from $\overline{H}_2(1,\alpha)=0, \overline{H}_2(2,\alpha)=1$. Now assume $k\geq 2$. Then, by induction and (\ref{h2barrecurrence}), we have
\begin{align*}
\overline{h}_2(2^k,\alpha)&=2\cdot\frac{2^{k-1}\overline{h}_2(2^{k-1},\alpha)}{2^k}+\frac{2^k\cdot2^{k-1}}{2^k\alpha^{2^k}}\\
&=\sum_{j=1}^{k-1}\frac{2^{j-1}}{\alpha^{2^j}}+\frac{2^{k-1}}{\alpha^{2^k}}=\sum_{j=1}^k\frac{2^{j-1}}{\alpha^{2^j}},
\end{align*}
as required.
\end{proof}

Motivated by this, let $g(k,\alpha)=\sum_{j=1}^k2^{j-1}/\alpha^{2^j}$ for every $\alpha\not=0$, and let $\lambda(\alpha)=\sum_{j=1}^\infty2^{j-1}/\alpha^{2^j}=\lim_{k\to\infty}g(k,\alpha)$ for every $\alpha>1$. Roughly speaking, the following lemma shows that $\overline{h_2}(m,\alpha)$ do not deviate much from $g(\floor{\log_2 m},\alpha)$. 

\begin{lemma}\label{tildeasymptotic}
For every $k\geq1$ and $\alpha\not=0$, let $\widetilde{g}(k,\alpha)=\min\{\overline{h}_2(m,\alpha)\colon2^k\leq m<2^{k+1}\}$. Then, $$\widetilde{g}(k+1,\alpha)\geq\min\left\{\frac{g(k,\alpha)}{3}+\frac{2\widetilde{g}(k,\alpha)}{3}, \frac{g(k+1,\alpha)}{2}+\frac{\widetilde{g}(k,\alpha)}{2}\right\}.$$
Therefore, for all $\alpha>1$, $\lim_{k\to\infty}\widetilde{g}(k,\alpha)=\lim_{k\to\infty}g(k,\alpha)=\liminf_{m\to\infty}\overline{h}_2(m,\alpha)=\lambda(\alpha)$.
\end{lemma}
\begin{proof}
Let $k\geq 1$ and let $2^{k+1}\leq m<2^{k+2}$. Note that $\widetilde{g}(k,\alpha)\leq\overline{h}_2(2^k,\alpha)=g(k,\alpha)$.

\noindent\textbf{Case 1.} If $2^{k+1}\leq m<2^k+2^{k+1}$, set $c=m-2^k$ and note that $2^k\leq c<2^{k+1}$. By (\ref{h2barrecurrence}), Lemma \ref{2knice} and the definition of $\widetilde{g}(k,\alpha)$, we have \begin{align*}
\overline{h}_2(m,\alpha)&\geq\frac{2^k\overline{h}_2(2^k,\alpha)}{m}+\frac{c\overline{h}_2(c,\alpha)}{m}=\frac{2^kg(k,\alpha)}{m}+\frac{c\overline{h}_2(c,\alpha)}{m}\\
&\geq\frac{2^kg(k,\alpha)}{m}+\frac{c\widetilde{g}(k,\alpha)}{m}\geq \frac{g(k,\alpha)}{3}+\frac{2\widetilde{g}(k,\alpha)}{3}.
\end{align*}

\noindent\textbf{Case 2.} If $2^k+2^{k+1}\leq m<2^{k+2}$, set $c=m-2^{k+1}$ and note that $2^k\leq c<2^{k+1}$. Again, by (\ref{h2barrecurrence}). Lemma \ref{2knice} and the definition of $\widetilde{g}(k)$, we have 
\begin{align*}
\overline{h}_2(m,\alpha)&\geq\frac{2^{k+1}\overline{h}_2(2^{k+1},\alpha)}{m}+\frac{c\overline{h}_2(c,\alpha)}{m}=\frac{2^{k+1}g(k+1,\alpha)}{m}+\frac{c\overline{h}_2(c,\alpha)}{m}\\
&\geq\frac{2^{k+1}g(k+1,\alpha)}{m}+\frac{c\widetilde{g}(k,\alpha)}{m}\geq \frac{g(k+1,\alpha)}{2}+\frac{\widetilde{g}(k,\alpha)}{2}.
\end{align*}
Hence, for all $k\geq1$, we have $\widetilde{g}(k+1,\alpha)=\min\{\overline{h}_2(m,\alpha)\colon2^{k+1}\leq m<2^{k+2}\}\geq\min\{\frac13g(k,\alpha)+\frac23\widetilde{g}(k,\alpha),\frac12g(k+1,\alpha)+\frac12\widetilde{g}(k,\alpha)\}$, as claimed. 

Moreover, it follows that for every $k\geq 1$, we have $\widetilde{g}(k+1,\alpha)-\widetilde{g}(k,\alpha)\geq\frac13(g(k,\alpha)-\widetilde{g}(k,\alpha))$ and $\widetilde{g}(k+1,\alpha)-\widetilde{g}(k,\alpha)\geq \frac12(g(k+1,\alpha)-\widetilde{g}(k,\alpha))$, both of them imply that $\widetilde{g}(k+1,\alpha)\geq \widetilde{g}(k,\alpha)$. When $\alpha>1$, this means $\widetilde{g}(k,\alpha)$ is increasing in $k$ and bounded above by $\lambda(\alpha)$, so $\lim_{k\to\infty}\widetilde{g}(k,\alpha)$ exists. Taking limit on both sides of $\widetilde{g}(k+1,\alpha)\geq\min\{\frac13g(k,\alpha)+\frac23\widetilde{g}(k,\alpha),\frac12g(k+1,\alpha)+\frac12\widetilde{g}(k,\alpha)\}$, we see that in fact $\lim_{k\to\infty}\widetilde{g}(k,\alpha)=\lim_{k\to\infty}g(k,\alpha)=\lambda(\alpha)$. Finally, $\liminf_{m\to\infty}\overline{h}_2(m,\alpha)=\lim_{m\to\infty}\inf_{n\geq m}\overline{h}_2(n,\alpha)=\lim_{k\to\infty}\inf_{n\geq 2^k}\overline{h}_2(n,\alpha)=\lim_{k\to\infty}\inf_{k'\geq k}\widetilde{g}(k',\alpha)=\lim_{k\to\infty}\widetilde{g}(k,\alpha)=\lambda(\alpha)$. 
\end{proof}

We can now quickly deduce Theorem \ref{thm:generallowerbds}.
\begin{proof}[Proof of Theorem \ref{thm:generallowerbds}]
For every $\alpha>1$, $\liminf_{n\to\infty}\overline{h}_2(m,\alpha)=\lambda(\alpha)$ by Lemma \ref{tildeasymptotic}. Therefore, for every $\epsilon>0$ and all sufficiently large $m$, we have $\overline{h}_2(m,\alpha)>\lambda(\alpha)-\epsilon$, and thus $H_2(m,\alpha)\geq\overline{H}_2(m,\alpha)>(\lambda(\alpha)-\epsilon)m\alpha^m/2$.

For every $r\geq 3$, apply the above to $\alpha_r=\frac{r-1}{r-2}$. Then, by Theorem \ref{main}, for every $\epsilon>0$ and all sufficiently large $m$ we have
\begin{align*}
\forb(m,r,M)&\geq(r-1)^m+m(r-1)^{m-1}+H_2(m,\alpha_r)(r-2)^{m-2}\\
&\geq(r-1)^m+m(r-1)^{m-1}+\overline{H}_2(m,\alpha_r)(r-2)^{m-2}\\
&\geq(r-1)^m+m(r-1)^{m-1}+m(r-1)^{m-1}(\lambda(\alpha_r)-\epsilon)\frac{r-1}{2(r-2)^2}\\
&\geq(r-1)^m+m(r-1)^{m-1}\left(1+\frac{r-1}{2(r-2)^2}\lambda\left(\frac{r-1}{r-2}\right)-\epsilon\right),
\end{align*}
as required.
\end{proof}

\subsection{Optimal splits}\label{h2exact}
In this subsection, we show that in the case when $\alpha=2$, the split we used in the lower bound construction in Section \ref{lowerboundconstruction} is actually optimal with the exception when $m=6$, which proves Theorem \ref{thm:lowerbds}. 

For simplicity, write $H_2(m)=H_2(m,2)$ and $h_2(m)=h_2(m,2)$. A table of values of $H_2(m)$ and $h_2(m)$, rounded to 3 decimals, for small $m$ is given below. 
\begin{table}[h]
\centering
\begin{tabular}{ |c|c|c|c|c|c|c|c|c| }
 \hline
 $m$ &1 & 2 & 3 & 4 & 5 & 6 & 7 & 8\\
\hline
 $H_2(m)$ & 0 & 1 & 4 & 12 & 30 & 73 & 172 & 400\\
 \hline
 $h_2(m)$ & 0.000 & 0.250 & 0.333 & 0.375 & 0.375 & 0.380 & 0.384 & 0.391\\
 \hline
\end{tabular}
\end{table}

\begin{thm}\label{optimalsplit}
For all $m\geq2$ except 6, we have $$h_2(m)=\frac{2^kh_2(2^k)}{m}+\frac{(m-2^k)h_2(m-2^k)}{m}+\frac{2^k(m-2^k)}{m2^{m-1}},$$ where $k$ is the unique integer such that $2^k+2^{k-1}\leq m<2^k+2^{k+1}$.  
\end{thm}
\begin{proof}
We use induction on $m$. It is easy to verify this using the recurrence relation of $h_2(m)$ and the initial values given above for all $m\leq48$. Hence, we only need to prove the theorem for all $k\geq5$ and all $2^{k-1}+2^k\leq m<2^k+2^{k+1}$.

Fix some $k\geq5$ and $2^{k-1}+2^k\leq m<2^k+2^{k+1}$. Assume inductively that the result holds for all smaller values of $k,m$. We need to show that for all positive integers $a\leq b$ satisfying $a+b=m$, we have
\begin{equation}\label{2to1ineqoriginal}
2^kh_2(2^k)+(m-2^k)h_2(m-2^k)+\frac{2^k(m-2^k)}{2^{m-1}}\geq ah_2(a)+bh_2(b)+\frac{ab}{2^{m-1}}.
\end{equation}
Let $c=m-2^k$. Depending on the values of $a,b,c,m$, we will split into twelve subcases. The proofs of all subcases follow the same general recipe, which will be described and demonstrated below. However, the calculations involved are long and somewhat tedious, so we will only show the full proofs for four of them and sketch the rest.  

We begin by introducing some terminologies and notations. For positive integers $x,y,n$, we say that $x$ and $y$ form a \textit{split} of $n$ if $x\leq y$ and $x+y=n$. In this case, by definition we have $nh_2(n)\geq xh_2(x)+yh_2(y)+xy/2^{n-1}$. We say that $x$ and $y$ form an \textit{optimal split} of $n$ if they form a split of $n$ and $nh_2(n)=xh_2(x)+yh_2(y)+xy/2^{n-1}$. We call a term of the form $xh_2(x)$ a \textit{major term}, and denote it by $[x]$. In contrast, a term of the form $xy/2^{n-1}$ (with $n=x+y$) is called a \textit{minor term}, and is denoted by $[x,y]$. In this language, if $x$ and $y$ form a split of $n$, then $[n]\geq[x]+[y]+[x,y]$, with equality if $x$ and $y$ form an optimal split. Our goal (\ref{2to1ineqoriginal}) is equivalent to 
\begin{equation}\label{2to1ineq}
[2^k]+[c]+[2^k,c]\geq[a]+[b]+[a,b]
\end{equation}

The general recipe for proving each subcase consists of two steps. We first focus on the major terms. The aim is to find a series of intermediate splits, often optimal ones, with the help of the induction hypothesis, so that when we substitute them into (\ref{2to1ineq}) the major terms all cancel out. Then we deal with the remaining minor terms. Some computations and estimates are needed here to show that we do have the desired inequalities, but across subcases the calculations all follow a similar pattern. 

We now list all the subcases and prove a few of them in detail. For the remaining ones, we will give the intermediate splits needed to cancel out the major terms, but leave out the minor terms calculations.

\noindent\textbf{Case 1. }$2^{k+1}\leq m<2^k+2^{k+1}$. Then $2^{k}\leq b<2^k+2^{k+1}$. 

\textbf{Case 1.1. } $2^{k-1}+2^k\leq b<2^k+2^{k+1}$. Then $a<2^{k-1}+2^k$ as $a+b=m<2^k+2^{k+1}$. There are two subcases.

\indent\indent\textbf{Case 1.1.1. }$a<2^k$. By induction hypothesis, $2^k$ and $b-2^k$ form an optimal split of $b$. Also $a+(b-2^k)=m-2^k=c$, so $a$ and $b-2^k$ form a split of $c$. Hence, we have  
\begin{align*}
[b]&=[2^k]+[b-2^k]+[2^k,b-2^k],\\
[c]&\geq [a]+[b-2^k]+[a,b-2^k].
\end{align*}
Substitute these into (\ref{2to1ineq}), we see that all the major terms cancel out and it suffices now to show that the remaining minor terms satisfy $[2^k,c]+[a,b-2^k]\geq [2^k,b-2^k]+[a,b]$. Expand and simplify, this is equivalent to showing $2^kc+2^{2^k}a(b-2^k)\geq2^k2^a(b-2^k)+ab$. To simplify notations, let $s=2^k$, $t=a$ and $r=b-2^k$. Then $0\leq t<s$, $s/2\leq r<2s$, $c=a+b-2^k=r+t$, and we need to show $s(r+t)+tr2^s\geq sr2^t+(s+r)t$. Equivalently, we need $(s-t)r+r(t2^s-s2^t)\geq0$. If $t=0$, both sides are 0. Otherwise, it follows from the fact that the sequence $(x_n)_{n=1}^{\infty}$ given by $x_n=2^n/n$ satisfies $x_1=x_2<x_3<x_4<\cdots$.

\indent\indent\textbf{Case 1.1.2. }$2^k\leq a< 2^{k-1}+2^k$. Then $b=m-a<2^{k+1}$. Also, we have $2^{k-1}+2^k\leq c=a+b-2^k=m-2^k<2^{k+1}$. Hence by induction hypothesis, we know optimal splits of $a,b,c$ and $2^k$, which gives
\begin{align*}
[a]&=[2^{k-1}]+[a-2^{k-1}]+[2^{k-1},a-2^{k-1}],\\
[b]&=[2^k]+[b-2^k]+[2^k,b-2^k],\\
[c]&=[2^k]+[c-2^k]+[2^k,c-2^k],\\
[2^k]&=[2^{k-1}]+[2^{k-1}]+[2^{k-1},2^{k-1}].
\end{align*}
Moreover, we have $2^k\leq c-2^{k-1}=(a-2^{k-1})+(b-2^k)<2^{k-1}+2^k$, so by induction hypothesis,
\begin{align*}
[c-2^{k-1}]&=[2^{k-1}]+[c-2^k]+[2^{k-1},c-2^k]\\
&\geq[a-2^{k-1}]+[b-2^k]+[a-2^{k-1},b-2^k].
\end{align*}
Substitute these into (\ref{2to1ineq}), we see that all the major terms cancel out, and it suffices to show that the remaining minor terms satisfy 
\begin{align*}
&\phantom{==}[2^k,c]+[c-2^k,2^k]+[2^{k-1},2^{k-1}]+[a-2^{k-1},b-2^k]\\
&\geq [a,b]+[b-2^k,2^k]+[a-2^{k-1},2^{k-1}]+[2^{k-1},c-2^k].
\end{align*}
Expand and simplify, this is equivalent to
\begin{align*}
&\phantom{==}c2^k+(c-2^k)2^k2^{2^k}+2^{2k-2}2^c+(a-2^{k-1})(b-2^k)2^{3\cdot2^{k-1}}\\
&\geq ab+(b-2^k)2^k2^a+(a-2^{k-1})2^{k-1}2^b+2^{k-1}(c-2^k)2^{3\cdot2^{k-1}}.
\end{align*}
We perform similar substitutions $s=2^{k-1}$, $t=a-2^k$ and $r=b-2^{k-1}-2^k$ to simplify notations. Note that $0\leq r,t<s$ and $c=3s+t+r$. Then it suffices to show
\begin{align*}
&\phantom{==}2s(3s+t+r)+2s(s+t+r)2^{2s}+s^22^{3s+t+r}+(s+t)(s+r)2^{3s}\\
&\geq (2s+t)(3s+r)+2s(s+r)2^{2s+t}+s(s+t)2^{3s+r}+s(s+t+r)2^{3s}.
\end{align*}
Expand and simplify, this is equivalent to
$$(s2^{2s+r}-(2s+r)2^s)(s2^{s+t}-(s+t)2^s)+2^{2s}r(t2^s-s2^t+s-t)\geq st+rt.$$
Again, each term on the left hand side is non-negative as the sequence $(x_n)=(2^n/n)$ is decreasing. If $t=0$, then both sides are 0. Otherwise, since $s=2^{k-1}\geq16$, the left hand side is at least $s2^s(2^s-3)\cdot1+0\geq2s^2\geq st+rt$, which proves this subcase. 

\textbf{Case 1.2. }$2^k\leq b<2^{k-1}+2^k$. Then we have $2^{k-1}\leq m-b=a\leq b<2^{k-1}+2^k$. Split into two subcases.

\indent\indent\textbf{Case 1.2.1. }$2^k\leq a<2^{k-1}+2^k$. By induction hypothesis, we have 
\begin{align*}
[a]&=[2^{k-1}]+[a-2^{k-1}]+[2^{k-1},a-2^{k-1}],\\
[b]&=[2^{k-1}]+[b-2^{k-1}]+[2^{k-1},b-2^{k-1}],\\
[2^k]&=[2^{k-1}]+[2^{k-1}]+[2^{k-1},2^{k-1}].
\end{align*}
Moreover, $(a-2^{k-1})+(b-2^{k-1})=m-2^k=c$, so $$[c]\geq[a-2^{k-1}]+[b-2^{k-1}]+[a-2^{k-1},b-2^{k-1}].$$
Substitute these into (\ref{2to1ineq}) and simplify, we see that it suffices to show $$2^c2^{2k-2}+2^kc+(a-2^{k-1})(b-2^{k-1})2^{2^k}\geq ab+(a-2^{k-1})2^{k-1}2^b+(b-2^{k-1})2^{k-1}2^a.$$
To simplify notations, let $s=2^{k-1}, t=a-s, r=b-s$. Then $s\leq r,t<2s$ and $c=t+r$. It now suffices to show $s^22^{t+r}+2s(t+r)+tr2^{2s}\geq(s+t)(s+r)+st2^{s+r}+sr2^{s+t}$. This is equivalent to $(2^ts-2^st)(2^rs-2^sr)\geq(t-s)(r-s)$. Since the function $(2^x+1)/x$ is increasing on $[2,\infty)$ and $t\geq s=2^{k-1}\geq16$, we have $(2^t+1)/t\geq(2^s+1)/s$. Hence $2^ts-2^st\geq t-s$. Similarly, $2^rs-2^sr\geq r-s$, which finishes the proof of this subcase. 

\indent\indent\textbf{Case 1.2.2. }$2^{k-1}\leq a<2^k$. Then $2^k\leq m-2^k=c=a+b-2^k<2^{k-1}+2^k$. So by induction hypothesis, we have
\begin{align*}
[b]&=[2^{k-1}]+[b-2^{k-1}]+[2^{k-1},b-2^{k-1}],\\
[c]&=[2^{k-1}]+[c-2^{k-1}]+[2^{k-1},c-2^{k-1}],\\
[2^k]&=[2^{k-1}]+[2^{k-1}]+[2^{k-1},2^{k-1}].
\end{align*}
Moreover, $2^{k-1}+2^{k}\leq c+2^{k-1}<2^{k+1}$ and $a+(b-2^{k-1})=c+2^{k-1}$. So 
\begin{align*}
[c+2^{k-1}]&=[2^k]+[c-2^{k-1}]+[2^k,c-2^{k-1}]\\
&\geq[a]+[b-2^{k-1}]+[a,b-2^{k-1}].
\end{align*}
Substitute these into (\ref{2to1ineq}), we see that it suffices to prove 
\begin{align*}
&\phantom{==}a(b-2^{k-1})2^{2^{k-1}}+(c-2^{k-1})2^{k-1}2^{2^k}+c2^k\\
&\geq(c-2^{k-1})2^k2^{2^{k-1}}+(b-2^{k-1})2^{k-1}2^a+ab.
\end{align*}
Let $s=2^{k-1},t=a-s,r=b-2s$. Then $c=s+t+r$ and $0\leq t,r<s$. Substitute and simplify, we see that it suffices to show $$s(t+r)2^{2s}+(s+t)(s+r)2^s+(s-t)r\geq s(s+r)2^{s+t}+2s(t+r)2^s.$$
Note that $(s-t)r\geq 0$, so after some rearrangement it suffices to show $$s(2^s-2)(r(2^s-2^t)+(t2^s-s2^t))+(s+r)2^{-s}((s+t)2^{2s}-2s2^{s+t})\geq0,$$
which is true as each term is non-negative, proving this subcase. 


\noindent\textbf{Case 2. }$2^{k-1}+2^k\leq m<2^{k+1}$. Then $2^{k-2}+2^{k-1}\leq b<2^{k+1}$. 

\textbf{Case 2.1. }$2^{k-1}+2^k\leq b<2^{k+1}$. Then $1\leq a<2^{k-1}$. Use 
\begin{align*}
[b]&=[2^k]+[b-2^k]+[2^k,b-2^k],\\
[c]&\geq[a]+[b-2^k]+[a,b-2^k].
\end{align*}

\textbf{Case 2.2. }$2^{k-2}+2^k\leq b<2^{k-1}+2^k$. Then $1\leq a<2^{k-2}+2^{k-1}$.

\indent\indent\textbf{Case 2.2.1. }$2^{k-1}\leq a<2^{k-2}+2^{k-1}$. Use
\begin{align*}
[a]&=[2^{k-2}]+[a-2^{k-2}]+[2^{k-2},a-2^{k-2}],\\
[b]&=[2^{k-1}]+[b-2^{k-1}]+[2^{k-1},b-2^{k-1}],\\
[b-2^{k-1}]&=[2^{k-1}]+[b-2^k]+[2^{k-1},b-2^k],\\
[c]&=[2^{k-1}]+[c-2^{k-1}]+[2^{k-1},c-2^{k-1}],\\
[2^k]&=[2^{k-1}]+[2^{k-1}]+[2^{k-1},2^{k-1}],\\
[2^{k-1}]&=[2^{k-2}]+[2^{k-2}]+[2^{k-2},2^{k-2}],\\
[c-2^{k-2}]&=[2^{k-2}]+[c-2^{k-1}]+[2^{k-2},c-2^{k-1}]\\
&\geq[a-2^{k-2}]+[b-2^k]+[a-2^{k-2},b-2^k].
\end{align*}

\indent\indent\textbf{Case 2.2.2. }$1\leq a<2^{k-1}$. Use
\begin{align*}
[b]&=[2^{k-1}]+[b-2^{k-1}]+[2^{k-1},b-2^{k-1}],\\
[b-2^{k-1}]&=[2^{k-1}]+[b-2^k]+[2^{k-1},b-2^k],\\
[2^k]&=[2^{k-1}]+[2^{k-1}]+[2^{k-1},2^{k-1}],\\
[c]&\geq[a]+[b-2^k]+[a,b-2^k].
\end{align*}

\textbf{Case 2.3. }$2^k\leq b<2^{k-2}+2^k$. Then $2^{k-2}\leq a<2^k$. 

\indent\indent\textbf{Case 2.3.1. }$2^{k-2}+2^{k-1}\leq a<2^k$. Use
\begin{align*}
[a]&=[2^{k-1}]+[a-2^{k-1}]+[2^{k-1},a-2^{k-1}],\\
[b]&=[2^{k-1}]+[b-2^{k-1}]+[2^{k-1},b-2^{k-1}],\\
[b-2^{k-1}]&=[2^{k-2}]+[b-3\cdot2^{k-2}]+[2^{k-2},b-3\cdot2^{k-2}],\\
[c]&=[2^{k-1}]+[c-2^{k-1}]+[2^{k-1},c-2^{k-1}],\\
[2^k]&=[2^{k-1}]+[2^{k-1}]+[2^{k-1},2^{k-1}],\\
[2^{k-1}]&=[2^{k-2}]+[2^{k-2}]+[2^{k-1},2^{k-2}],\\
[c-2^{k-2}]&=[2^{k-2}]+[c-2^{k-1}]+[2^{k-2},c-2^{k-1}]\\
&\geq[a-2^{k-1}]+[b-3\cdot2^{k-2}]+[a-2^{k-1},b-3\cdot2^{k-2}].
\end{align*}

\indent\indent\textbf{Case 2.3.2. }$2^{k-1}\leq a<2^{k-2}+2^{k-1}$. Use \begin{align*}
[a]&=[2^{k-2}]+[a-2^{k-2}]+[2^{k-2},a-2^{k-2}],\\
[b]&=[2^{k-1}]+[b-2^{k-1}]+[2^{k-1},b-2^{k-1}],\\
[b-2^{k-1}]&=[2^{k-2}]+[b-3\cdot2^{k-2}]+[2^{k-2},b-3\cdot2^{k-2}],\\
[2^k]&=[2^{k-1}]+[2^{k-1}]+[2^{k-1},2^{k-1}],\\
[2^{k-1}]&=[2^{k-2}]+[2^{k-2}]+[2^{k-1},2^{k-2}],\\
[c]&\geq[a-2^{k-2}]+[b-3\cdot2^{k-2}]+[a-2^{k-2},b-3\cdot2^{k-2}].
\end{align*}

\indent\indent\textbf{Case 2.3.3. }$2^{k-2}\leq a<2^{k-1}$. Use
\begin{align*}
[b]&=[2^{k-1}]+[b-2^{k-1}]+[2^{k-1},b-2^{k-1}],\\
[2^k]&=[2^{k-1}]+[2^{k-1}]+[2^{k-1},2^{k-1}],\\
[c+2^{k-1}]&=[2^{k-1}]+[c]+[c,2^{k-1}]\\
&\geq[a]+[b-2^{k-1}]+[a,b-2^{k-1}].
\end{align*}

\textbf{Case 2.4. }$2^{k-2}+2^{k-1}\leq b<2^k$. Then $2^{k-1}\leq a<2^k$.

\indent\indent\textbf{Case 2.4.1. }$2^{k-2}+2^{k-1}\leq a<2^k$. Use \begin{align*}
[a]&=[2^{k-1}]+[a-2^{k-1}]+[2^{k-1},a-2^{k-1}],\\
[b]&=[2^{k-1}]+[b-2^{k-1}]+[2^{k-1},b-2^{k-1}],\\
[2^k]&=[2^{k-1}]+[2^{k-1}]+[2^{k-1},2^{k-1}],\\
[c]&\geq[a-2^{k-1}]+[b-2^{k-1}]+[a-2^{k-1},b-2^{k-1}].
\end{align*}

\indent\indent\textbf{Case 2.4.2. }$2^{k-1}\leq a<2^{k-2}+2^{k-1}$. Use
\begin{align*}
[a]&=[2^{k-2}]+[a-2^{k-2}]+[2^{k-2},a-2^{k-2}],\\
[b]&=[2^{k-1}]+[b-2^{k-1}]+[2^{k-1},b-2^{k-1}],\\
[2^k]&=[2^{k-1}]+[2^{k-1}]+[2^{k-1},2^{k-1}],\\
[c]&=[2^{k-2}]+[c-2^{k-2}]+[2^{k-2},c-2^{k-2}],\\
[c+2^{k-2}]&=[2^{k-1}]+[c-2^{k-2}]+[2^{k-1},c-2^{k-2}]\\
&\geq[a-2^{k-2}]+[b-2^{k-1}]+[a-2^{k-2},b-2^{k-1}].\qedhere
\end{align*}
\end{proof}

Theorem \ref{optimalsplit} proves the first part of Theorem \ref{thm:lowerbds} with the exception of $m=6$. To prove the second part about the limiting behaviour of $H_{2}(m)$, we need the following lemma. Let $g(k)=g(k,2)=\sum_{j=1}^k2^{j-1}/2^{2^{j}}$ and $\lambda=\lambda(2)=\sum_{j=1}^{\infty}2^{j-1}/2^{2^{j}}$.

\begin{lemma}\label{fleqg}
For all positive integers $m$ except $3,6$ and $7$, we have $$h_2(m)\leq g(\floor{\log_2m}),$$
with equality if and only if $m=5$ or $m$ is a power of $2$. In particular, $\limsup_{m\to\infty}h_2(m)=\lambda$.
\end{lemma}
\begin{proof}
We will use induction on $m$. Using the recurrence relation given by Theorem \ref{optimalsplit} and the initial values of $h_2(m)$ given in the table above, it is easy to verify the claim for all $1\leq m\leq23$. Now assume $m\geq 24$ and the results holds all smaller $m$. Let $k$ be the unique integer such that $2^{k-1}+2^k\leq m<2^k+2^{k+1}$, and note that $k\geq4$. Let $c=m-2^k\geq8$. By Theorem \ref{optimalsplit}, it suffices to show that
\begin{equation*}
2^kh_2(2^k)+ch_2(c)+\frac{2^kc}{2^{m-1}}\leq mg(\floor{\log_2m}).
\end{equation*} 

\noindent\textbf{Case 1.} $2^{k-1}+2^k\leq m< 2^{k+1}$. Then $\floor{\log_2m}=k$ and $\floor{\log_2c}=k-1$. By induction hypothesis, we have
\begin{align*}
2^kh_2(2^k)+ch_2(c)+\frac{2^kc}{2^{m-1}}&\leq2^kg(k)+cg(k-1)+\frac{2^kc}{2^{m-1}}\\
&=mg(k-1)+2^k\frac{2^{k-1}}{2^{2^k}}+\frac{2^kc}{2^{m-1}}\\
&\leq mg(k-1)+2^k\frac{2^{k-1}}{2^{2^k}}+c\frac{2^{k-1}}{2^{2^k}}\\
&=mg(k).
\end{align*}
Note that the first inequality is strict by induction hypothesis unless $c=2^{k-1}$. But in that case the second inequality is strict.

\noindent\textbf{Case 2.} $2^{k+1}\leq m<2^k+2^{k+1}$. Then $\floor{\log_2m}=k+1$ and $\floor{\log_2c}=k$. By induction hypothesis, we have
\begin{align*}
2^kh_2(2^k)+ch_2(c)+\frac{2^kc}{2^{m-1}}&\leq2^kg(k)+cg(k)+\frac{2^kc}{2^{m-1}}\\
&=mg(k)+\frac{2^kc}{2^{m-1}}\\
&\leq mg(k)+m\frac{2^k}{2^{2^{k+1}}}\\
&=mg(k+1).
\end{align*}
Note that the first inequality is strict unless $c=2^k$. In that case both inequalities are tight and $m=2^{k+1}$ is a power of 2. 

Therefore, $\limsup_{m\to\infty}h_2(m)\leq\lim_{k\to\infty}g(k)=\lambda$. Then, as $h_2(2^k)=g(k)$ for every $k\geq 1$, we must have $\limsup_{m\to\infty}h_2(m)=\lambda$.
\end{proof}

We can now put everything together and prove Theorem \ref{thm:lowerbds}.
\begin{proof}[Proof of Theorem \ref{thm:lowerbds}]
The first part of Theorem \ref{thm:lowerbds} when $m\not=6$ follows from $h_2(m)=H_2(m)/m2^{m-1}$ and Theorem \ref{optimalsplit}, while for $m=6$, it is easy to verify that the optimal split is $6=3+3$.

For the second part, Lemma \ref{fleqg} gives that $\limsup_{m\to\infty}h_2(m)=\lambda=\lambda(2)$, while Lemma \ref{tildeasymptotic} implies $\liminf_{m\to\infty}h_2(m)\geq\liminf_{m\to\infty}\overline{h}_2(m,2)=\lambda(2)$. Hence, $\lim_{m\to\infty}H_2(m)/m2^{m-1}=\lim_{m\to\infty}h_2(m)$ exists and is equal to $\lambda(2)$.
\end{proof}

\section{Further results and open questions}\label{closingremarks}
By suitably modifying Section \ref{uppersection} and Section \ref{lowersection}, we can also prove the following generalisation of Theorem \ref{thm:upperbds} and Theorem \ref{thm:lowerbds} for all $\alpha\geq 2$. 
\begin{thm}\label{generalupper}
For every $m\geq 3$ and $\alpha\geq 2$, we have $$H(m,\alpha)\leq\frac{\alpha^6+\alpha^4+\alpha+1}{\alpha^6(\alpha^2-1)}\cdot\frac{m\alpha^m}2.$$    
\end{thm}
\begin{thm}\label{generallower}
For every $m\geq 3$ except 6, let $k=k(m)$ be the unique integer such that $2^k+2^{k-1}\leq m<2^k+2^{k+1}$. For $m=6$, let $k=k(6)=3$ if $2\leq\alpha<\sqrt{2}+1$, and let $k=k(6)=4$ if $\alpha\geq\sqrt2+1$. Then for all $\alpha\geq 2$, we have $$H_{2}(m,\alpha)=H_{2}(2^k,\alpha)\alpha^{m-2^k}+H_{2}(m-2^k,\alpha)\alpha^{2^k}+2^k(m-2^k).$$
Moreover, if $\lambda(\alpha)=\sum_{j=1}^{\infty}2^{j-1}/\alpha^{2^{j}}$, then $\lim_{m\to\infty}\frac{2H_{2}(m,\alpha)}{m\alpha^m}=\lambda(\alpha)$.
\end{thm}
We also conjecture that the following generalised version of Conjecture \ref{conj} holds as well. 
\begin{conj}
$H(m,\alpha)=H_{2}(m,\alpha)$ for all $m\geq 3$ and $\alpha\geq 2$. In particular,
$$\lim_{m\to\infty}\frac{2H(m,\alpha)}{m\alpha^m}=\lim_{m\to\infty}\frac{2H_{2}(m,\alpha)}{m\alpha^m}=\lambda(\alpha).$$
\end{conj}
Observe that $\frac{\alpha^6+\alpha^4+\alpha+1}{\alpha^6(\alpha^2-1)}\sim\alpha^{-2}$ and $\lambda(\alpha)\sim\alpha^{-2}$ as $\alpha\to\infty$, which lends some support for our conjecture. One possible approach towards proving this conjecture is to show that for every $[m]$, there exists an extremal TCM on $[m]$ with exactly two maximal closed sets, as that would imply that there is always an extremal TCM that is 2-recursive. To this end, it seems possible to use similar arguments to those in the proof of Proposition \ref{prop:sizeatleast2}, to either show that every vertex in an extremal TCM must be contained in a closed set of increasingly larger size, or show that if an extremal TCM is partitioned into at least three closed sets, at least two of them can be combined into a larger closed set.  

In the cases when $r\geq 4$ and $1<\alpha<2$, the arguments proving Theorem \ref{generalupper} fail at some point and we do not have a good conjecture on when $\forb(m,r,M),H(m,\alpha)$ are attained, although part of the argument proving Theorem \ref{generallower} still works for sufficiently large $m$, so it is plausible that we still have $\lim_{m\to\infty}2H_2(m,\alpha)/m\alpha^m=\lambda(\alpha)$. The values of $H(m,\alpha)$ and $H_2(m,\alpha)$ when $\alpha<1$ are also of independent interest, but we haven't investigated them much as they do not correspond to forbidden matrix problems. 

\bibliographystyle{abbrv}
\bibliography{Citations.bib}
\end{document}

%% file: ExampleTCM.tex
\begin{subfigure}[b]{0.9\linewidth}
    \centering
    \begin{tikzpicture}
    \draw 
        (0, 4) node[circle, black, draw](a){1}
        (0, 2) node[circle, black, draw](b){2}
        (0, 0) node[circle, black, draw](c){3}
        (2, 1) node[circle, black, draw](d){4}
        (2, 3) node[circle, black, draw](e){5};

        \draw [-] (a) to node {} (b);
        \draw [-] (a) to [bend right=20] node {} (b);
        \draw [-] (a) to [bend left=20] node {} (b);

        \draw [-] (a.220) to [bend right=30] node {} (c.120);
        \draw [-] (a.210) to [bend right=60] node {} (c.130);

        \draw [-] (b) to node {} (c);
        \draw [-] (b) to [bend left=20] node {} (c);

        \draw [-] (d) to node {} (e);
        \draw [-] (d) to [bend right=20] node {} (e);
        \draw [-] (d) to [bend left=20] node {} (e);

    \end{tikzpicture}

\end{subfigure}\\
\newline

\begin{subfigure}[b]{\textwidth}
\centering
\begin{tabular}{|l|L{0.05\textwidth}|L{0.05\textwidth}|L{0.05\textwidth}|L{0.06\textwidth}|L{0.06\textwidth}|L{0.05\textwidth}|L{0.05\textwidth}|L{0.05\textwidth}|L{0.05\textwidth}|L{0.05\textwidth}|}
\hline
\makecell{Vertex\\Triplet} & 1,2,3 & 1,2,4 & 1,2,5 & 1,3,4 & 1,3,5 & 1,4,5 & 2,3,4 & 2,3,5 & 2,4,5 & 3,4,5 \\
\hline 
Subgraph &
\begin{subfigure}[b]{0.05\textwidth}
    \centering
    \medskip
    \begin{tikzpicture}
    \draw 
    (0, 0.8) node[circle, red, draw, scale = 0.4](a){1}
    (0, 0.4) node[circle, red, draw, scale = 0.4](b){2}
    (0, 0) node[circle, red, draw, scale = 0.4](c){3}
    (0.5, 0.6) node[circle, black, draw, scale = 0.4](d){4}
    (0.5, 0.2) node[circle, black, draw, scale = 0.4](e){5};
    \draw [-] (a) to node[below] {} (b);
    \end{tikzpicture}
\end{subfigure} &
\begin{subfigure}[b]{0.05\textwidth}
    \centering
    \medskip
    \begin{tikzpicture}
    \draw 
    (0, 0.8) node[circle, red, draw, scale = 0.4](a){1}
    (0, 0.4) node[circle, red, draw, scale = 0.4](b){2}
    (0, 0) node[circle, black, draw, scale = 0.4](c){3}
    (0.5, 0.6) node[circle, red, draw, scale = 0.4](d){4}
    (0.5, 0.2) node[circle, black, draw, scale = 0.4](e){5};
    \draw [-] (a) to node[below] {} (b);
    \end{tikzpicture}
\end{subfigure} &
\begin{subfigure}[b]{0.05\textwidth}
    \centering
    \medskip
    \begin{tikzpicture}
    \draw 
    (0, 0.8) node[circle, red, draw, scale = 0.4](a){1}
    (0, 0.4) node[circle, red, draw, scale = 0.4](b){2}
    (0, 0) node[circle, black, draw, scale = 0.4](c){3}
    (0.5, 0.6) node[circle, black, draw, scale = 0.4](d){4}
    (0.5, 0.2) node[circle, red, draw, scale = 0.4](e){5};
    \draw [-] (a) to node[below] {} (b);
    \end{tikzpicture}
\end{subfigure} &
\begin{subfigure}[b]{0.06\textwidth}
    \centering
    \medskip
    \begin{tikzpicture}
    \draw 
    (0, 0.8) node[circle, red, draw, scale = 0.4](a){1}
    (0, 0.4) node[circle, black, draw, scale = 0.4](b){2}
    (0, 0) node[circle, red, draw, scale = 0.4](c){3}
    (0.5, 0.6) node[circle, red, draw, scale = 0.4](d){4}
    (0.5, 0.2) node[circle, black, draw, scale = 0.4](e){5};
    \draw [-] (a) to [bend right=40] node[below] {} (c);
    \end{tikzpicture}
\end{subfigure} &
\begin{subfigure}[b]{0.06\textwidth}
    \centering
    \medskip
    \begin{tikzpicture}
    \draw 
    (0, 0.8) node[circle, red, draw, scale = 0.4](a){1}
    (0, 0.4) node[circle, black, draw, scale = 0.4](b){2}
    (0, 0) node[circle, red, draw, scale = 0.4](c){3}
    (0.5, 0.6) node[circle, black, draw, scale = 0.4](d){4}
    (0.5, 0.2) node[circle, red, draw, scale = 0.4](e){5};
    \draw [-] (a) to [bend right=40] node[below] {} (c);
    \end{tikzpicture}
\end{subfigure} &
\begin{subfigure}[b]{0.05\textwidth}
    \centering
    \medskip
    \begin{tikzpicture}
    \draw 
    (0, 0.8) node[circle, red, draw, scale = 0.4](a){1}
    (0, 0.4) node[circle, black, draw, scale = 0.4](b){2}
    (0, 0) node[circle, black, draw, scale = 0.4](c){3}
    (0.5, 0.6) node[circle, red, draw, scale = 0.4](d){4}
    (0.5, 0.2) node[circle, red, draw, scale = 0.4](e){5};
    \draw [-] (d) to node[below] {} (e);
    \end{tikzpicture}
\end{subfigure} &
\begin{subfigure}[b]{0.05\textwidth}
    \centering
    \medskip
    \begin{tikzpicture}
    \draw 
    (0, 0.8) node[circle, black, draw, scale = 0.4](a){1}
    (0, 0.4) node[circle, red, draw, scale = 0.4](b){2}
    (0, 0) node[circle, red, draw, scale = 0.4](c){3}
    (0.5, 0.6) node[circle, red, draw, scale = 0.4](d){4}
    (0.5, 0.2) node[circle, black, draw, scale = 0.4](e){5};
    \draw [-] (b) to node[below] {} (c);
    \end{tikzpicture}
\end{subfigure} &
\begin{subfigure}[b]{0.05\textwidth}
    \centering
    \medskip
    \begin{tikzpicture}
    \draw 
    (0, 0.8) node[circle, black, draw, scale = 0.4](a){1}
    (0, 0.4) node[circle, red, draw, scale = 0.4](b){2}
    (0, 0) node[circle, red, draw, scale = 0.4](c){3}
    (0.5, 0.6) node[circle, black, draw, scale = 0.4](d){4}
    (0.5, 0.2) node[circle, red, draw, scale = 0.4](e){5};
    \draw [-] (b) to node[below] {} (c);
    \end{tikzpicture}
\end{subfigure} &
\begin{subfigure}[b]{0.05\textwidth}
    \centering
    \medskip
    \begin{tikzpicture}
    \draw 
    (0, 0.8) node[circle, black, draw, scale = 0.4](a){1}
    (0, 0.4) node[circle, red, draw, scale = 0.4](b){2}
    (0, 0) node[circle, black, draw, scale = 0.4](c){3}
    (0.5, 0.6) node[circle, red, draw, scale = 0.4](d){4}
    (0.5, 0.2) node[circle, red, draw, scale = 0.4](e){5};
    \draw [-] (d) to node[below] {} (e);
    \end{tikzpicture}
\end{subfigure} &
\begin{subfigure}[b]{0.05\textwidth}
    \centering
    \medskip
    \begin{tikzpicture}
    \draw 
    (0, 0.8) node[circle, black, draw, scale = 0.4](a){1}
    (0, 0.4) node[circle, black, draw, scale = 0.4](b){2}
    (0, 0) node[circle, red, draw, scale = 0.4](c){3}
    (0.5, 0.6) node[circle, red, draw, scale = 0.4](d){4}
    (0.5, 0.2) node[circle, red, draw, scale = 0.4](e){5};
    \draw [-] (d) to node[below] {} (e);
    \end{tikzpicture}
\end{subfigure}\\
\hline
\end{tabular}
\end{subfigure}
    
    \caption{An example triangular choice multigraph $\mathcal{G}$ on 5 vertices, which is 2-recursive with $V_1 = \{1,2,3\}$ and $V_2 = \{4,5\}$. All 10 vertex triplets are enumerated in the table with the chosen edges indicated.}
    \label{fig:ExampleTCM}

%% file: ExampleZeroImplications.tex
%
%
%
\begin{subfigure}[b]{0.9\linewidth}
    \centering
    \begin{tikzpicture}
    \draw 
        (0, 6) node[circle, black, draw](a){1}
        (0, 3) node[circle, black, draw](b){2}
        (0, 0) node[circle, black, draw](c){3}
        (8, 4.5) node[circle, black, draw](d){4}
        (8, 1.5) node[circle, black, draw](e){5};

        \Edge[label=\scalebox{0.75}{\color{red} ${1,2,3}$},style={->, pos=0.5, thin, sloped}](a)(b)
        \Edge[label=\scalebox{0.75}{\color{red} ${1,2,3}$},style={->, pos=0.6, thin, bend right=30, sloped}](a)(c)

        \Edge[label=\scalebox{0.75}{\color{red} $1,2,4$},style={->, pos=0.5, thin, sloped, bend left=16}](a)(d);
        \Edge[label=\scalebox{0.75}{\color{red} $1,2,4$},style={->, pos=0.225, thin, sloped, bend left=15.5}](b)(d);

        \Edge[label=\scalebox{0.75}{\color{red} $1,2,5$},style={->, pos=0.225, thin, sloped, bend left=12.5}](a)(e);
        \Edge[label=\scalebox{0.75}{\color{red} $1,2,5$},style={->, pos=0.275, thin, sloped, bend left=7}](b)(e);

        \Edge[label=\scalebox{0.75}{\color{red} $1,3,4$},style={->, pos=0.5, thin, sloped, bend left=8}](a)(d);
        \Edge[label=\scalebox{0.75}{\color{red} $1,3,4$},style={->, pos=0.225, thin, bend left=7.5, sloped}](c)(d);

        \Edge[label=\scalebox{0.75}{\color{red} $1,3,5$},style={->, pos=0.225, thin, sloped, bend left = 2.5}](a)(e);
        \Edge[label=\scalebox{0.75}{\color{red} $1,3,5$},style={->, pos=0.5, thin, sloped, bend left=0}](c)(e);

        \Edge[label=\scalebox{0.75}{\color{red} $1,4,5$},style={->, pos=0.5, thin, bend left=0, sloped}](a)(d);
        \Edge[label=\scalebox{0.75}{\color{red} $1,4,5$},style={->, pos=0.225, thin, bend right=7.5, sloped}](a)(e);

        \Edge[label=\scalebox{0.75}{\color{red} $2,3,4$},style={->, pos=0.25, thin, bend left=4, sloped}](b)(d);
        \Edge[label=\scalebox{0.75}{\color{red} $2,3,4$},style={->, pos=0.225, thin, sloped, bend right=2.5}](c)(d);

        \Edge[label=\scalebox{0.75}{\color{red} $2,3,5$},style={->, pos=0.25, thin, bend right=4, sloped}](b)(e);
        \Edge[label=\scalebox{0.75}{\color{red} $2,3,5$},style={->, pos=0.5, thin, sloped, bend right=8}](c)(e);

        \Edge[label=\scalebox{0.75}{\color{red} $2,4,5$},style={->, pos=0.275, thin, bend right=7, sloped}](b)(d);
        \Edge[label=\scalebox{0.75}{\color{red} $2,4,5$},style={->, pos=0.225, thin, bend right=15.5, sloped}](b)(e);

        \Edge[label=\scalebox{0.75}{\color{red} $3,4,5$},style={->, pos=0.225, thin, bend right=12.5, sloped}](c)(d);
        \Edge[label=\scalebox{0.75}{\color{red} $3,4,5$},style={->, pos=0.5, thin, bend right=16, sloped}](c)(e);

    \end{tikzpicture}

\end{subfigure}\\
\newline

\begin{subfigure}[b]{\textwidth}
\centering
\begin{tabular}{|C{0.07\textwidth}|C{0.07\textwidth}|C{0.07\textwidth}|C{0.07\textwidth}|C{0.07\textwidth}|C{0.07\textwidth}|C{0.07\textwidth}|C{0.07\textwidth}|C{0.07\textwidth}|C{0.07\textwidth}|}
\hline
    \small{$B_{1,2,3}$} & \small{$B_{1,2,4}$} & \small{$B_{1,2,5}$} & \small{$B_{1,3,4}$} & \small{$B_{1,3,5}$} & \small{$B_{1,4,5}$} & \small{$B_{2,3,4}$} & \small{$B_{2,3,5}$} & \small{$B_{2,4,5}$} & \small{$B_{3,4,5}$}\\
    \hline
    $A_1$ & $A_2$ & $A_2$ & $A_2$ & $A_2$ & $A_1$ & $A_2$ & $A_2$ & $A_1$ & $A_1$\\
    \hline
    \begin{subfigure}[b]{0.06\textwidth}
        \centering
        \medskip
        \begin{tikzpicture}
        \draw 
        (0, 0.75) node[circle, black, draw, scale = 0.4](a){1}
        (0, 0.4) node[circle, black, draw, scale = 0.4](b){2}
        (0, 0) node[circle, black, draw, scale = 0.4](c){3}
        (0.5, 0.6) node[circle, black, draw, scale = 0.4](d){4}
        (0.5, 0.2) node[circle, black, draw, scale = 0.4](e){5};

        \draw [->] (a) to node[below] {} (b);
        \draw [->] (a) to [bend right=40] node[below] {} (c);
        \draw [-, red] (b) to node[below] {} (c);
        
        \end{tikzpicture}
    \end{subfigure} &
    \begin{subfigure}[b]{0.06\textwidth}
        \centering
        \medskip
        \begin{tikzpicture}
        \draw 
        (0, 0.75) node[circle, black, draw, scale = 0.4](a){1}
        (0, 0.4) node[circle, black, draw, scale = 0.4](b){2}
        (0, 0) node[circle, black, draw, scale = 0.4](c){3}
        (0.5, 0.6) node[circle, black, draw, scale = 0.4](d){4}
        (0.5, 0.2) node[circle, black, draw, scale = 0.4](e){5};

        \draw [->] (b) to node[below] {} (d);
        \draw [->] (a) to node[below] {} (d);
        \draw [-, red] (a) to node[below] {} (b);
        
        \end{tikzpicture}
    \end{subfigure} &
    \begin{subfigure}[b]{0.06\textwidth}
        \centering
        \medskip
        \begin{tikzpicture}
        \draw 
        (0, 0.75) node[circle, black, draw, scale = 0.4](a){1}
        (0, 0.4) node[circle, black, draw, scale = 0.4](b){2}
        (0, 0) node[circle, black, draw, scale = 0.4](c){3}
        (0.5, 0.6) node[circle, black, draw, scale = 0.4](d){4}
        (0.5, 0.2) node[circle, black, draw, scale = 0.4](e){5};

        \draw [-, red] (a) to node[below] {} (b);
        \draw [->] (a) to node[below] {} (e);
        \draw [->] (b) to node[below] {} (e);
        
        \end{tikzpicture}
    \end{subfigure} &
    \begin{subfigure}[b]{0.06\textwidth}
        \centering
        \medskip
        \begin{tikzpicture}
        \draw 
        (0, 0.75) node[circle, black, draw, scale = 0.4](a){1}
        (0, 0.4) node[circle, black, draw, scale = 0.4](b){2}
        (0, 0) node[circle, black, draw, scale = 0.4](c){3}
        (0.5, 0.6) node[circle, black, draw, scale = 0.4](d){4}
        (0.5, 0.2) node[circle, black, draw, scale = 0.4](e){5};
        
        \draw [-, red] (a) to [bend right=40] node[below] {} (c);
        \draw [->] (a) to node[below] {} (d);
        \draw [->] (c) to node[below] {} (d);
        
        \end{tikzpicture}
    \end{subfigure} &
    \begin{subfigure}[b]{0.06\textwidth}
        \centering
        \medskip
        \begin{tikzpicture}
        \draw 
        (0, 0.75) node[circle, black, draw, scale = 0.4](a){1}
        (0, 0.4) node[circle, black, draw, scale = 0.4](b){2}
        (0, 0) node[circle, black, draw, scale = 0.4](c){3}
        (0.5, 0.6) node[circle, black, draw, scale = 0.4](d){4}
        (0.5, 0.2) node[circle, black, draw, scale = 0.4](e){5};
        
        \draw [-, red] (a) to [bend right=40] node[below] {} (c);
        \draw [->] (a) to node[below] {} (e);
        \draw [->] (c) to node[below] {} (e);
        
        \end{tikzpicture}
    \end{subfigure} &
    \begin{subfigure}[b]{0.06\textwidth}
        \centering
        \medskip
        \begin{tikzpicture}
        \draw 
        (0, 0.75) node[circle, black, draw, scale = 0.4](a){1}
        (0, 0.4) node[circle, black, draw, scale = 0.4](b){2}
        (0, 0) node[circle, black, draw, scale = 0.4](c){3}
        (0.5, 0.6) node[circle, black, draw, scale = 0.4](d){4}
        (0.5, 0.2) node[circle, black, draw, scale = 0.4](e){5};

        \draw [->] (a) to node[below] {} (d);
        \draw [->] (a) to node[below] {} (e);
        \draw [-, red] (d) to node[below] {} (e);
        
        \end{tikzpicture}
    \end{subfigure} &
    \begin{subfigure}[b]{0.06\textwidth}
        \centering
        \medskip
        \begin{tikzpicture}
        \draw 
        (0, 0.75) node[circle, black, draw, scale = 0.4](a){1}
        (0, 0.4) node[circle, black, draw, scale = 0.4](b){2}
        (0, 0) node[circle, black, draw, scale = 0.4](c){3}
        (0.5, 0.6) node[circle, black, draw, scale = 0.4](d){4}
        (0.5, 0.2) node[circle, black, draw, scale = 0.4](e){5};

        \draw [-, red] (b) to node[below] {} (c);
        \draw [->] (b) to node[below] {} (d);
        \draw [->] (c) to node[below] {} (d);
        
        \end{tikzpicture}
    \end{subfigure} &
    \begin{subfigure}[b]{0.06\textwidth}
        \centering
        \medskip
        \begin{tikzpicture}
        \draw 
        (0, 0.75) node[circle, black, draw, scale = 0.4](a){1}
        (0, 0.4) node[circle, black, draw, scale = 0.4](b){2}
        (0, 0) node[circle, black, draw, scale = 0.4](c){3}
        (0.5, 0.6) node[circle, black, draw, scale = 0.4](d){4}
        (0.5, 0.2) node[circle, black, draw, scale = 0.4](e){5};

        \draw [-, red] (b) to node[below] {} (c);
        \draw [->] (b) to node[below] {} (e);
        \draw [->] (c) to node[below] {} (e);
        
        \end{tikzpicture}
    \end{subfigure} &
    \begin{subfigure}[b]{0.06\textwidth}
        \centering
        \medskip
        \begin{tikzpicture}
        \draw 
        (0, 0.75) node[circle, black, draw, scale = 0.4](a){1}
        (0, 0.4) node[circle, black, draw, scale = 0.4](b){2}
        (0, 0) node[circle, black, draw, scale = 0.4](c){3}
        (0.5, 0.6) node[circle, black, draw, scale = 0.4](d){4}
        (0.5, 0.2) node[circle, black, draw, scale = 0.4](e){5};

        \draw [->] (b) to node[below] {} (d);
        \draw [->] (b) to node[below] {} (e);
        \draw [-, red] (d) to node[below] {} (e);
        
        \end{tikzpicture}
    \end{subfigure} &
    \begin{subfigure}[b]{0.06\textwidth}
        \centering
        \medskip
        \begin{tikzpicture}
        \draw 
        (0, 0.75) node[circle, black, draw, scale = 0.4](a){1}
        (0, 0.4) node[circle, black, draw, scale = 0.4](b){2}
        (0, 0) node[circle, black, draw, scale = 0.4](c){3}
        (0.5, 0.6) node[circle, black, draw, scale = 0.4](d){4}
        (0.5, 0.2) node[circle, black, draw, scale = 0.4](e){5};

        \draw [->] (c) to node[below] {} (d);
        \draw [->] (c) to node[below] {} (e);
        \draw [-, red] (d) to node[below] {} (e);
        
        \end{tikzpicture}
    \end{subfigure}\\
    \hline
\end{tabular}
\end{subfigure}

    \caption{An example directed multigraph $\mathcal{D}_{\mathcal{B}}$ on 5 vertices showing the 0-implications from a good (and uniformly directed) choice $\mathcal{B}$. The 0-implications corresponding to each $B_{i,j,k} \in \mathcal{B}$ are enumerated in the table. Each directed edge is labeled with the triple that contributed the corresponding 0-implication.}
    \label{fig:ExampleZeroImplications}

%% file: ExampleClosedSets.tex
    \centering
    \begin{tikzpicture}
    \draw 
        (0, 0) node[circle, black, draw](a){1}
        (2.5, 0) node[circle, black, draw](b){2}
        (5, 0) node[circle, black, draw](c){3}
        (7.5, 0) node[circle, black, draw](d){4}
        (10, 0) node[circle, black, draw](e){5};

    \draw [-] (a) -- node[above] {\small{$m_{12} = 3$}} (b);
    \draw [-] (b) -- node[above] {\small{$m_{23}=2$}} (c);
    \draw [-] (d) -- node[above] {\small{$m_{45}=3$}} (e);

    \draw [-] (a) to [bend right=30] node[below] {$\small{m_{13}=2}$} (c);

    \draw[red, dashed] (2.5,0) ellipse (3.5 and 1.4);
    \draw[densely dotted] (1.25,0) ellipse (1.8 and 0.85);
    \draw[red, dashed] (8.75,0) ellipse (1.9 and 0.9);
    \draw[densely dotted] (4.75, 0) ellipse (6.1 and 1.8);

    \end{tikzpicture}    
    \caption{An example TCM $G$ with its closed sets identified. There are two maximal closed sets, identified in dashed red, with other non-maximal closed sets being singletons (trivially closed) or identified in dotted black.} 
    \label{fig:ExampleClosedSets}